\documentclass[hyperref]{article}
\usepackage{graphicx} 
\usepackage{float} 
\usepackage{amssymb}
\usepackage{amsmath}
\usepackage{mathtools}
\usepackage{enumitem}
\newcommand{\pl}{\partial}
\usepackage[thmmarks,amsmath]{ntheorem} 
\usepackage{bm} 

\usepackage[colorlinks,linkcolor=cyan,citecolor=cyan]{hyperref}

\usepackage{extarrows}
\usepackage{epstopdf}

\usepackage{amsmath}
\numberwithin{equation}{section}

\usepackage{graphicx}
\usepackage{subfigure}

\usepackage[hang,flushmargin]{footmisc} 
\usepackage[square,comma,sort&compress,numbers]{natbib} 
\usepackage{mathrsfs} 
\usepackage{booktabs} 

{
	\theoremstyle{nonumberplain}
	\theoremheaderfont{\bfseries}
	\theorembodyfont{\normalfont}
	\theoremsymbol{\mbox{$\Box$}}
	\newtheorem{proof}{Proof}
}

\usepackage{theorem}
\newtheorem{theorem}{Theorem}[section]
\newtheorem{lemma}{Lemma}[section]
\newtheorem{definition}{Definition}[section]

{
	\theoremheaderfont{\bfseries}
	\theorembodyfont{\normalfont}
	
}
\graphicspath{{figure/}}                                 
\usepackage[a4paper]{geometry}                           
\begin{document}
	\title{\bf BINet: Learning to Solve Partial Differential Equations with Boundary Integral Networks} 
	\date{}
	\author{\sffamily Guochang Lin$^1$, Pipi Hu$^{1,3}$, Fukai Chen$^2$, Xiang Chen$^{4}$, \\\sffamily Junqing Chen$^{2}$, Jun Wang$^{5}$, Zuoqiang Shi$^{2,3*}$ \\
		{\sffamily\small $^1$ Yau Mathematical Sciences Center, Tsinghua University, Beijing, China}\\
		{\sffamily\small $^2$ Department of Mathematical Sciences, Tsinghua University, Beijing, China}\\
		{\sffamily\small $^3$ Yanqi Lake Beijing Institute of Mathematical Sciences and Applications, Beijing, China}\\
		{\sffamily\small $^4$ Noah’s Ark Lab, Huawei, China}\\
		{\sffamily\small $^5$ University College London, London, United Kingdom}}
	\renewcommand{\thefootnote}{\fnsymbol{footnote}}
	\footnotetext[1]{Corresponding author. \emph{Email address: }zqshi@tsinghua.edu.cn (Z. Shi)}
	\maketitle
	
	{\noindent\small{\bf Abstract:}
		We propose a method combining boundary integral equations and neural networks (BINet) to solve partial differential equations (PDEs) in both bounded and unbounded domains. 
		Unlike existing solutions that directly operate over original PDEs, BINet learns to solve, as a proxy, associated boundary integral equations using neural networks. 
		The benefits are three-fold. 
		Firstly, only the boundary conditions need to be fitted since the PDE can be automatically satisfied with single or double layer representations according to the potential theory.
		Secondly, the dimension of the boundary integral equations is typically smaller, and as such, the sample complexity can be reduced significantly. Lastly,  in the proposed method, all differential operators of the original PDEs have been removed, hence the numerical efficiency and stability are improved. 
		Adopting neural tangent kernel (NTK) techniques, 
		we provide proof of the convergence of BINets in the limit that the width of the neural network goes to infinity. Extensive numerical experiments show that, without calculating high-order derivatives, BINet is much easier to train and usually gives more accurate solutions, especially in the cases that the boundary conditions are not smooth enough. Further, BINet outperforms strong baselines for both one single PDE and parameterized PDEs in the bounded and unbounded domains.

	}
	
	\section{Introduction}
	\label{Introduction}
	Partial differential equations (PDEs) have been widely used in 
	scientific fields and engineering applications, such as Maxwell's equations in optics and electromagnetism \cite{griffiths2005introduction},
	Navier–Stokes equations in fluid dynamics \cite{temam2001navier}, the Schr\"{o}dinger equations in the quantum physics \cite{schrodinger1926undulatory}, and Black-Scholes equations for call option pricing in finance \cite{macbeth1979empirical}. Therefore, finding the solution to PDEs has been a critical topic in research over the years. However, in most cases, the analytical solution of PDEs is infeasible to obtain, such that numerical methods become the major bridge between PDE models and practical applications.
	
	In the past decade, deep learning has achieved great success in computer vision, natural language processing, and many other topics \cite{goodfellow2016deep}. It is found that deep neural networks (DNNs) have the attractive capability in approximating functions, especially in high dimensional space. Therefore, DNNs hold great potential in solving PDEs with the promise of providing a good \emph{ansatz} to represent the solution, where the parameters can be obtained by training DNNs with proper loss functions.
	
	In the literature, many efforts have been devoted to developing DNN-based methods for solving different kinds of PDEs, 
	such as DGM \cite{sirignano2018dgm}, Deep-Ritz \cite{weinan2018deep}, and PINN \cite{raissi2019physics}. The main idea of these methods is to use a neural network to approximate the solution of the PDE directly. The loss function is designed by either incorporating the PDE residual and the boundary or initial conditions, or the energy functional derived from the variational form of the PDE. 
	
	However, two important issues are not fully considered in most existing works. First, PDEs are merely utilized to construct the loss function, and the essence behind PDEs may be further explored to design a new network structure to cater the need for solving differential equations. Second, when it comes to complex problems, such as PDEs with oscillatory or even singular solutions, failure of the aforementioned methods is frequently reported \cite{wang2020and} due to high order differentiation of the neural networks with respect to the inputs. The appearance of high-order derivatives may lead to instability in training \cite{zhu2021local} (for example, amplified oscillation or singularities) such that the network can not find the exact solution. 

	To address the two issues above, in this paper, we propose a novel method, named BINet, combining boundary integral equations and deep learning to solve PDEs. Utilizing
	fundamental solutions of PDE and Green's formula \cite{kellogg1953foundations}, the solution to PDE can be expressed in the form of a boundary integral, where the explicit fundamental solution of PDE serves as the integral kernel. A new network structure is then designed based on this integral expression of the solution such that the output of our network can satisfy the PDE automatically. Since the PDE has been satisfied, we only need to take the boundary condition as the supervisory signal for the loss. In BINet, the prior information provided by the PDE is fully integrated into the network. Moreover, the differential operator is substituted by an integral operator, which avoids the extra differential operations of the neural networks. The main advantages of BINet are summarized below:

	First, BINet adopts an explicit integral representation of the solution such that the output of BINet satisfies the original PDE automatically. This means that the training of BINet is naturally confined in the solution space of PDE. Since BINet is defined in a much smaller space, i.e., the solution function space, the training of BINet is faster and more stable than the general neural network. Another advantage of integral representation is that all differential operators are removed in BINet. Then the regularity requirement of BINet is relaxed significantly which enables BINet to approximate the solutions with poor regularity. Moreover, BINet has good theoretical properties. Using neural tangent kernel (NTK) techniques \cite{jacot2018neural}, BINet can be proved to converge as the width of the neural network goes to infinity.

	Second, since the PDE has been satisfied automatically with the integral representation in BINet, the residual of the boundary condition is the only component of the loss function. There is no need to balance the residual of the PDE and the boundary condition, BINet thus fits the boundary condition better with less parameter tuning.

	Third, BINet can solve PDEs in the unbounded domain since the integral representation holds for both bounded and unbounded domains. For some problems such as electromagnetic wave propagation, solving PDEs in an unbounded domain is critical and complicated using traditional methods. Moreover, existing deep-learning-based models also suffer from the difficulty of sampling in unbounded domains. Therefore, BINet provides a good choice to solve this kind of problem.

	Fourth, BINet is also capable to learn a solution operator mapping a parameterized PDE to its solution by feeding the parameters to the network as input. Note that in the integral representation of the solution, the integral kernel, i.e., the fundamental solution to the original PDE, has an explicit form dependent on the differential operator of the PDE. Moreover, the integral is conducted exactly on the boundary of the domain on which the PDE is solved. Therefore, BINet has great advantages in learning the solution operator mapping differential operator or computational domain to the corresponding solution.
	
	At last, the boundary integral is defined on the boundary whose dimension is less by 1 than the original computational domain. Lower dimension leads to fewer sample points which will reduce the computational cost.
	
	The rest of this paper is organized as follows.  An overview of related work on solving PDEs using deep learning approaches is given in Section \ref{sec:rel}. The boundary integral method and BINet are introduced in Section \ref{sec:BINet}. In Section \ref{sec:ntk}, we analyze the convergence of BINet using the NTK techniques. Extensive numerical experiments are shown in Section \ref{sec:num}. At last, conclusion remarks are made in Section \ref{sec:con}.

	\section{Related Work}\label{sec:rel}
	Solving PDEs with neural network can be traced back to 1990s \cite{dissanayake1994neural, lagaris1998artificial,lagaris2000neural}.
	Together with the deep learning revolution, solving PDEs with neural networks also enter a period of prosperity. In a neural network-based PDE solver, the loss function and network structure are two key ingredients. 
	
	Regarding the loss function,  
	one natural choice is the residual of PDE. In \cite{raissi2019physics, sirignano2018dgm}, $L_2$ norm of the residual is used as the loss function. For elliptic equations, the variation form provides another choice of the loss function. Yu and E proposed to use Ritz variational form as the loss function in \cite{weinan2018deep} and Galerkin variational form was formulated as an adversarial problem in \cite{zang2020weak}. In \cite{cai2020deep, lyu2020mim}, to avoid high order derivatives in the loss function, high order PDEs are first transformed to first-order PDEs system by introducing auxiliary variables. For the first-order system, we only need to compute first-order derivatives in the loss function.
	To solve PDEs, boundary condition has to be imposed properly. One simple way to enforce the boundary condition is to add it to the loss function as a penalty term. 
	In this approach, we must tune a weight to balance the PDEs' residual and boundary conditions. Usually, this weight is crucial and subtle to get good results. 
	The other way is to impose the boundary condition explicitly by introducing a distance function of the boundary \cite{berg2018unified}. Regarding that network structure, there are also many works recently. A fully connected neural network (FCN) is one of the most frequently used networks. In \cite{weinan2018deep}, it is found out that residual neural network (ResNet) gives better results. For PDEs with multiscale structure, a multiscale neural network was designedspecifically by introducing multiscale structure in the network \cite{cai2019multi}. Activation function is another important part of neural networks. Choice of the activation function is closely related to the smoothness of the neural network. To compute high-order derivatives, smooth activation functions, sigmoid, tanh, etc., are often used in PDE solvers. ReLU activation which is used most often in machine learning is hardly used due to its poor regularity. For special PDEs, other activation functions are also used, such as sReLU \cite{cai2019multi}, sine function\cite{sitzmann2020implicit}. By constrast, our BINet adopts an explicit integral representation of the solution, therefore the output satisfies the original PDE automatically.

	Another related research is to learn the solution operator, i.e. map from the parameter space to the solution space. Both the parameter space and solution space may be infinitely dimensional. Therefore, learning solution operator is more challenging than solving a single PDE. Solution operators may be complicated also, and network architecture becomes more important. In \cite{gin2020deepgreen,li2020fourier}, while Green's function and Fourier transform are used respectively to design good network architecture, the purpose is not for solving PDEs and thus different from ours. The network solving the single PDE can also be generalized to learn solution operators \cite{cai2021deepm,gao2020phygeonet,lu2019deeponet}. In \cite{khoo2017solving,yang2018physics,yang2019adversarial}, a neural network is used to solve the PDEs with uncertainty.

	\section{Boundary Integral Network (BINet)}\label{sec:BINet}
	Let 
	$\Omega\subset\mathbb{R}^d$ be a bounded domain, $\overline{\Omega}$ be the closure of $\Omega$ and $\Omega^c=\mathbb{R}^d\backslash \overline{\Omega}$.
	We consider the PDE in the following form,
	\begin{equation}
		\begin{aligned}
			\mathcal{L}u(x) = 0.
			\label{main_eq}
		\end{aligned}
	\end{equation}
	In this paper, $\mathcal{L}$ is chosen to be Laplace operator $ -\Delta=-\sum_{i=1}^d\frac{\partial^2}{\partial x_i^2}$ or Helmholtz operator $-\Delta-k^2$. But in general, BINet can be applied as long as the fundamental solution of $\mathcal{L}$ in $\mathbb{R}^d$ can be obtained. We list more options of $\mathcal{L}$ in Appendix.
	We consider both interior problems and exterior problems. In interior and exterior problems, the PDE $\mathcal{L}u(x) = 0$ is defined in $\Omega$ and $\Omega^c$ respectively.
	
	In this paper, we consider the Dirichlet type of boundary condition $ u|_{\partial\Omega}=g(x).$ Other types of boundary conditions can be easily handled in BINet with a small modification of the boundary integral equation.

	\subsection{Potential Theory}
	In this subsection, we briefly introduce the basics of the potential theory which provides the theoretical foundation of BINet. 
	We recall an important theorem in potential theory \cite{kellogg1953foundations}.
	\begin{theorem}\label{thm.potential}
		For any continuous function $h$ defined on $\partial\Omega$, the single layer potential is defined as
		\begin{equation}
			\mathcal{S}[h](x):=-\int_{\partial\Omega} G(x,y) h(y) d s_{\boldsymbol{y}}, \label{Single}  
		\end{equation}
		and the double layer potential is defined as
		\begin{equation}
			\mathcal{D}[h](x):=-\int_{\partial\Omega} \frac{\partial G(x,y) }{\partial \boldsymbol{n}_{\boldsymbol{y}}} h(y) d s_{\boldsymbol{y}}.
			\label{Double}
		\end{equation}
		with $\boldsymbol{n}_{\boldsymbol{y}}$ denotes out normal of $\partial \Omega$ at $\boldsymbol{y}$, $G(x,y)$ is the fundamental solution of equation \eqref{main_eq}. Then, both single layer potential and double layer potential satisfy (\ref{main_eq}). 
		And for all $x_0\in\partial\Omega$, we have
		\begin{equation}\label{thm:Property}
			\begin{aligned}
				\lim_{x\to x_0}\mathcal{S}[h](x) & = \mathcal{S}[h](x_0),\\
				\lim_{x\to x_0^\pm}\mathcal{D}[h](x) & = \mathcal{D}[h](x_0)\mp\frac{1}{2}h(x),
			\end{aligned}
		\end{equation}
		where $x\to x_0^-$ and $x\to x_0^+$ mean converging in $\Omega$ and $\Omega^c$ respectively.
	\end{theorem}
	For many important PDEs, fundamental solutions can be written explicitly. For the Laplace equation $-\Delta u(x)=0$ in $\mathbb{R}^2$, the fundamental solution is $G(x,y)=-\frac{1}{2\pi}\text{ln}|x-y|$, while the fundamental solution for the Helmholtz equation $-\Delta u(x)-k^2u(x)=0$ in $\mathbb{R}^2$ is $G(x,y)=\frac{i}{4}H_0^1(k|x-y|)$ where $H^1_0$ is the Hankel function. For the Laplace equation and the Helmholtz equation in the high dimensional case and more equations, please refer to Appendix.

	Based on Theorem \ref{thm.potential}, the single/double layer potential \eqref{Single} \eqref{Double} give explicit integral representations for the solution of the PDE. Using these integral representations, we can construct a network such that the output of the network solves the PDE automatically even with random initialization. This is also the main observation in BINet.

	\subsection{The Structure of BINet}
	In this subsection, we will explain how to use the boundary integral form $\mathcal{S}[h](x, \theta)$ or $\mathcal{D}[h](x, \theta)$ to construct the structure of BINet. 
	As shown in Fig. \ref{fig.sd}, BINet consists of three components: \emph{input, approximation, integration}.
	\begin{itemize}
		\item From the integral formula of the single/double layer potential, it is clear that BINet has three inputs: point in the computational domain $x\in \Omega$, differential operator $\mathcal{L}_\alpha$,  and domain boundary $\partial \Omega_{\beta}$. Differential operator$\mathcal{L}_\alpha$ determines the fundamental solution $G$ and domain boundary $\partial \Omega_{\beta}$ gives the domain of the integral.    
		\item In the single/double layer potential, only a density function $h$ is unknown. In BINet, the density function $h$ is approximated using a multilayer perceptron (MLP) (or a residual network, a.k.a, ResNet) denoted as $h(y,\theta)$ with the learning parameter $\theta$. Note that $h$ is defined on the boundary only. 
		\item Compute single or double layer potential in Theorem \ref{thm.potential} by kernel integration of the density function on the boundary $\partial\Omega_{\beta}$ where the kernel is given by the explicit fundamental solution $G$. The integration can be done numerically by the methods shown in \cite{alpert1999hybrid, kapur1997high}.
	\end{itemize}
	
	\begin{figure}[t]
		\centering
		\includegraphics[width=0.9\textwidth]{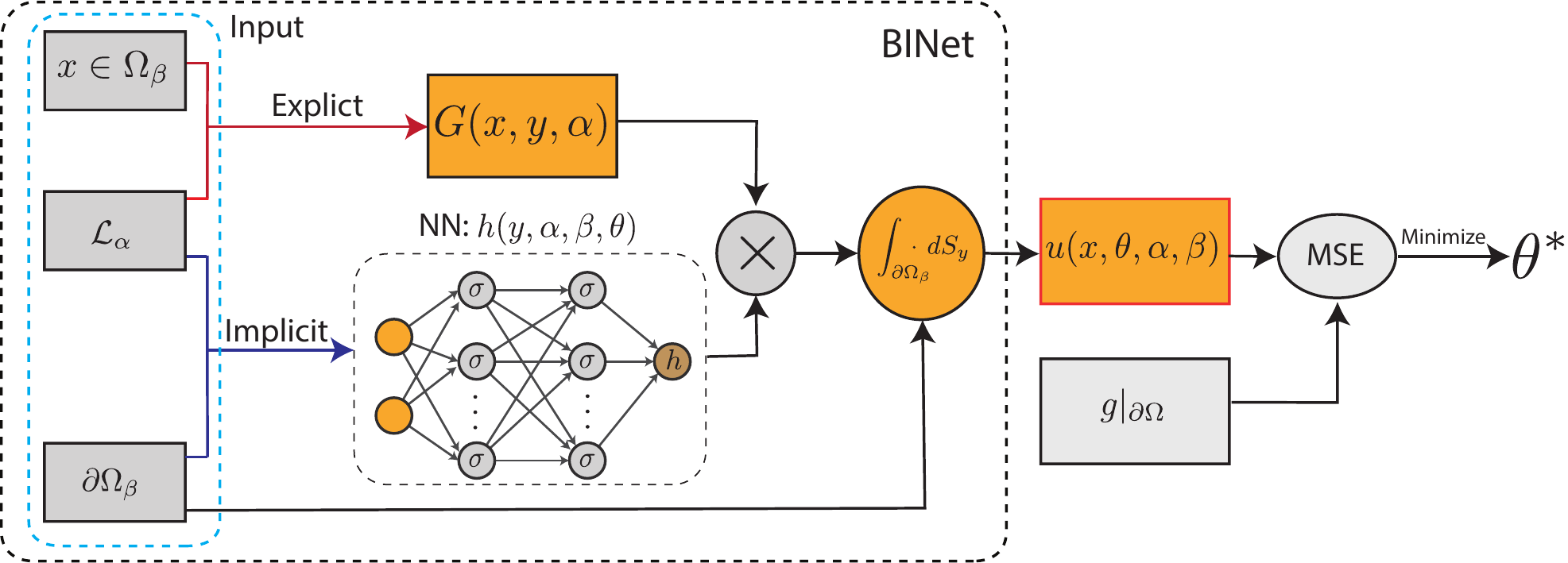}
		\caption{In BINet, fundamental solution $G(x,y,\alpha)$ explicitly depends on the operator $\mathcal{L}_\alpha$ of the equation, while the density function $h(y,\alpha,\beta,\theta)$ is implicitly dependent on $\mathcal{L}_\alpha$ and the boundary $\partial\Omega_\beta$. This implicit dependence is approximated by a neural network whose input is $\mathcal{L}_\alpha$ and $\partial \Omega_\beta$. These two parts are multiplied together and integrated on the boundary $\partial\Omega_\beta$, giving the output of BINet. The boundary condition is taken as the supervisory signal for the loss.
		}
		
		\label{fig.sd}
	\end{figure}    
	
	To train BINet, the loss function is given by \eqref{thm:Property} in Theorem \ref{thm.potential}.
	\begin{equation}\label{eq.loss}
		L(\theta) = \begin{cases}
			\|\mathcal{S}[h(\ \cdot \ ;\theta)](x)-g(x)\|^2_{\partial\Omega}, & \text{ single layer potential}\\[0.5em]
			\|(\frac{1}{2}\mathcal{I}+\mathcal{D})[h(\ \cdot \ ;\theta)](x)-g(x)\|^2_{\partial\Omega}, &\text{ double layer potential (Interior problem)}\\[0.5em]
			\|(-\frac{1}{2}\mathcal{I}+\mathcal{D})[h(\ \cdot \ ;\theta)](x)-g(x)\|^2_{\partial\Omega}, &\text{ double layer potential (Exterior problem)}
		\end{cases}
	\end{equation}
	where $\mathcal{S}$ and $\mathcal{D}$ are the potential operators defined in Theorem \ref{thm.potential}, and $\mathcal{I}$ is the identity operator.

	In BINet, the differential operator $\mathcal{L}_\alpha$ and the computational domain boundary $\partial \Omega_\beta$ are naturally incorporated, which means that BINet has the capability to learn the map from the differential operator and computational domain to solutions.

	\section{Convergence Analysis of BINet}
	\label{sec:ntk}
	
	In recent years, many efforts have been devoted to the development of the convergence theory for the over-parameterized neural networks. In \cite{jacot2018neural}, a neural tangent kernel (NTK) is proposed to prove the convergence, and this tangent kernel is also implicit in these works \cite{du2019gradient, du2018gradient, li2018learning}. Later, a non-asymptotic proof using NTK is given in \cite{arora2019exact}. It is shown that a sufficiently wide network that has been fully trained is indeed equivalent to a kernel regression predictor. In this work, we give a non-asymptotic proof of the convergence for our BINet.
	
	In BINet, the density function in the boundary integral form is approximated by a neural network as $h(y,\theta)$. And a boundary integral operator is performed on the density function, giving the output of BINet on the boundary as $v(x) =\mathcal{A}[h](x,\theta), x\in\partial\Omega$. Here $\mathcal{A}=\mathcal{S}$ for the single layer potential and $\mathcal{A}=\pm \mathcal{I}/2 +\mathcal{D}$ for the double layer potential of the interior problem or the exterior problem. For simplicity, we denote $\mathcal{A}[f](x) = \int_{\partial\Omega}\tilde{G}(x,y)f(y)dy, x\in\partial\Omega$ as the output of BINet limited on the boundary. And the loss is given by the difference between the output and the boundary values, see Section \ref{sec:BINet} for detail. Due to the operator $\mathcal{A}$, the convergence analysis of this structure is non-trivial. 
	
	In the learning process, the evolution of the difference between the output and the boundary value obeys the following ordinary differential equation
	\begin{equation}\label{eq.converge}
		\begin{split}
			\frac{d}{dt}(v(x,\theta(t))-\tilde{v}(x)) & = 
			-\int_{\partial\Omega} (v(x',\theta(t))-\tilde{v}(x')) \mathcal{N}_t(x,x')dx'
		\end{split}
	\end{equation}
	where $v(x,\theta) = \mathcal{A}[h](x,\theta)$ is the output of BINet limited on the boundary and $\tilde{v}(x)$ is the boundary value, i.e., the label function. For a detailed derivation of \eqref{eq.converge}, see Appindex.
	Here $\mathcal{N}_t(x,x')=\sum_{\theta_p}\mathcal{A}[\frac{\partial h}{\partial \theta_p(t)}](x)\mathcal{A}[\frac{\partial h}{\partial \theta_p(t)}](x')$ is the kernel at training-step index $t$, with an admissible operator $\mathcal{A}$, see Appendix for detail.
	
	In the following two theorems, we would show that the kernel in \eqref{eq.converge} converges to a constant kernel independent of $t$ when the width of the layers goes to infinity. And the proof sketch is listed in the Appendix based on the works in \cite{arora2019exact}.
	
	\begin{theorem}\text{(Convergence result of kernel at initialization)}\
		Fix $\epsilon>0$ and $\delta\in(0,1)$. Suppose the activation nonlinear function $\sigma(\cdot) = \max\{\cdot,0\}$ is ReLU, the minimum width of the hidden layer $\text{min}_{l\in[L]}d_l\geq \Omega(\frac{L^6}{\epsilon^4}\log(L/\delta))$, and the operator $\mathcal{A}$ is bounded with $\|\mathcal{A}\|_\infty\leq A$. Then for the normalized data $x$ and $x'$ where $\|x\|\leq1$ and $\|x'\|\leq1$, with probability at least $1-\delta$
		we have $$|\mathcal{N}(x,x')-[\mathcal{A}\Theta^{(L)}\mathcal{A}](x,x')|\leq (L+1)A^2\epsilon.$$
		Here $[\mathcal{A}\Theta\mathcal{A}](x,x')$ is the constant kernel of BINet given by the neural-network kernel $\Theta(y,y')$. The front and the back operator means the operations are performed with the respect to the first and the second variable of the neural-network kernel. 
		\label{th.c1}
	\end{theorem}
	
	\begin{theorem}\text{(Convergence result of kernel during training)}
		Fix $\omega\leq \text{poly}(\frac{1}{L}, \frac{1}{n}, 1/\log(1/\delta), \lambda_0)$ and $\delta\in(0,1)$. Suppose that $\text{min}_l \{d_l\}\geq \text{poly}(1/\omega)$, and the operator $\mathcal{A}$ is bounded with $\|\mathcal{A}\|_\infty\leq A$. Then with probability at least $1-\delta$ over Gaussian random initialization, we have for all $t\geq 0$, $$|\mathcal{N}_t(x,x')-\mathcal{N}_0(x,x')|\leq A^2\omega,$$
		where $\mathcal{N}_t(x,x')$ is the kernel along time $t$ and $\mathcal{N}_0(x,x')\overset{\Delta}{=}\mathcal{N}(x,x')$ is the kernel when initialization over random Gaussian denoted in Theorem \ref{th.c1} to distinguish with the training process.
		\label{th.c2}
	\end{theorem}
	
	Further, we have the following lemma for the positive definiteness of the new constant kernel.
	\begin{lemma}\label{lemma.definite}
		$[\mathcal{A}\Theta^{(L)}\mathcal{A}](x,x')$ is positive definite for double layer potential in BINet. For single layer potential, the positive definiteness depends on the $C^\infty$ compactness of the boundary $\partial\Omega$.
	\end{lemma}
	The proof of Lemma \ref{lemma.definite} is given in Appendix. And the invertibility of the operator $\mathcal{A}$ is utilized to complete the proof. \cite{gao1991layer, verchota1984layer}
	
	By Lemma \ref{lemma.definite}, equation \eqref{eq.converge}, Theorem \ref{th.c1} and \ref{th.c2}, the error in BINet thus vanishes for double layer potential after fully training ($t\to\infty$) under the assumption that the the width of the neural network goes to infinity. And for single layer potential, the convergence results depend on the boundary, i.e., $\partial\Omega$ is $C^\infty$ compact. The proof of the convergence results is in the real space, however with the complex form of the kernel $[\mathcal{A}^*\Theta^{(L)}\mathcal{A}](x,x')$ with $\mathcal{A}^*[f](x) = \int_{\partial\Omega}\tilde{G}^*(x,y)f(y)dy$ for Helmholtz equations, the results still hold with inner product defined in the complex space.

	\section{Experiments}
	\label{sec:num}
	
	We use BINet to compute a series of examples including solving a single PDE, where differential operator and domain geometry are fixed, and learning solution operators. PDEs defined on both bounded and unbounded domains will be considered. In order to estimate the accuracy of the numerical solution $u$, the relative $L^2$ error $\Vert u-u^{\ast}\Vert_2/\Vert u^{\ast}\Vert_2$ is used, where $u^{\ast}$ is the exact solution. We compare our method with two state-of-the-art methods, the Deep Ritz method and PINN only for interior problems, since as we claimed before, other deep-learning-based PDE solvers are not able to handle exterior problems.
	

	
	In BINet, the fully connected neural network (MLP) or residual neural network (ResNet) are used to approximate the density function. Since there is no regularity requirement on density function, we can use any activation functions including ReLU. For the Laplace equation, the network only has one output, i.e., the approximation of density $h$, while for the Helmholtz equation, because its solution is complex, the network has two outputs, i.e., the real part and the imaginary part of density $h$. 
	In the experiments, we choose the Adam optimizer to minimize the loss function and all experiments are run on a single GPU of GeForce RTX 2080 Ti.

	\subsection{Experimental Results on Solving One Single PDE}
	\textbf{Laplace Equation with Smooth Boundary Condition.} 
	First, we consider a Laplacian equation in the bounded domain,
	\begin{equation}
		\begin{aligned}
			-\Delta u(x,y) = 0, \ (x,y)\in\Omega,\\
			u(x,y) = e^{ax}\sin(ay), \ (x,y)\in\partial\Omega,
		\end{aligned}
		\label{Ex1_eq}
	\end{equation}
	where $a$ is a fixed constant. We will compare the results of PINN, Deep-Ritz method, and BINet for different $a$. For simplicity, we choose $\Omega=[-1,1]\times[-1,1]$. In this example, we will use a residual neural network introduced in \cite{weinan2018deep}. We follow \cite{weinan2018deep} to choose $\text{ReLU}^3$ as the activation function in Deep Ritz method and PINN. In BINet, we use ReLU as the activation function since BINet has less regularity requirement. 
	
	When $a=4$, for these three methods, we all selected 800 equidistant sample points on $\partial\Omega$, and for PINN and Deep-Ritz method, we randomly selected 1600 sample points in $\Omega$. We all use residual neural networks with 40 neurons per layer and six blocks.
	
	When $a=8$, for the BINet method, we selected 2000 equidistant sample points on the boundary. For PINN and Deep-Ritz method, we randomly selected 4000 sample points in $\Omega$ and randomly selected 800 sample points on $\partial\Omega$. We also use residual neural networks with 100 neurons per layer and six blocks. But if we look at the solutions on $[-0.1,0.1]\times[-1,1]$, we find that the solutions of PINN and Deep Ritz method are quite different from the exact solution, but BINet method still captures the subtle structure of the exact solution. The results of different methods including PINN, Deep-Ritz method and BINet for $a=8$ are shown in Figure \ref{Ex1}.
	
	\begin{table}[H]
		\caption{Relative $L^2$ error of equation \eqref{Ex1_eq} with different methods.}
		\label{ex1_table}
		\centering
		\begin{tabular}{cccc}
			\toprule
			& PINN     & Deep Ritz     & {BINet} \\
			\midrule
			$a=4$ & 0.0140 & 0.0952  & \textbf{0.0031}     \\
			$a=8$ & 0.0262 & 0.2194 & \textbf{0.0002}      \\
			
			\bottomrule
		\end{tabular}
	\end{table}

	After training for 20000 epochs, the relative $L^2$ error of these methods is shown in the table \ref{ex1_table}.  
	In this example, with the same number of layers and neurons, BINet is always better than the other two methods no matter what the value of $a$ is. When $a$ increases, unlike other methods, the result of the BINet does not get worse.
	\noindent
	\begin{figure}[t]
		\centering
		\includegraphics[width=0.9\textwidth]{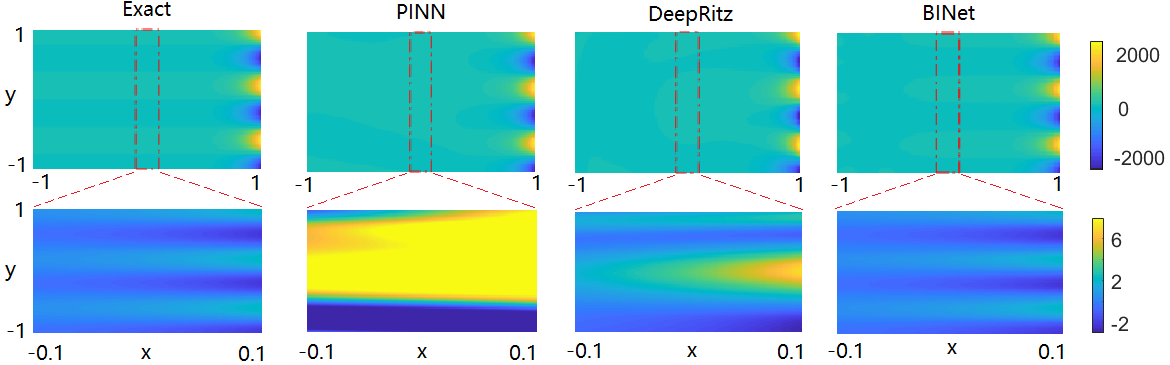}
		\caption{The solutions of Laplace equation \eqref{Ex1_eq} with smooth boundary conditions for $a=8$ by PINN, Deep-Ritz method and BINet. The first row shows the exact solutions and the numerical solutions obtained by the three methods in $\Omega=[-1,1]\times[-1,1]$. The second row is the zoom-in of the above figures in the subdomain $[-0.1,0,1]\times[-1,1]$. Only BINet captures the subtle structure of the exact solution successfully.}
		\label{Ex1}
	\end{figure}

	\noindent \textbf{Laplace Equation with Non-smooth Boundary Condition.} 
	Next, let's consider a Laplace equation with a nonsmooth boundary condition. We also assume the domain $\Omega=[-1,1]\times[-1,1]$ and the boundary value problem is
	\begin{equation}
		\begin{aligned}
			-\Delta u(x,y)=0,\ (x,y)\in\Omega,\\
			u(x,y) = 2-|x|-|y|, \ (x,y)\in\partial\Omega.
		\end{aligned}
		\label{Ex2_eq}
	\end{equation}
	In problem \eqref{Ex2_eq}, the boundary condition is not smooth. In this example, we also used the ResNet with six blocks and 40 neurons per layer for three methods. We selected equidistant 800 sample points on $\partial\Omega$ for three methods, and for PINN and Deep-Ritz method, we randomly selected 1000 sample points in $\Omega$. Figure \ref{f2} shows the results of different methods. In this example, we take the result of the finite difference method with high precision mesh as the exact solution.  
	
	\begin{figure}[t]
		\centering
		\centering
		\includegraphics[width=0.9\textwidth]{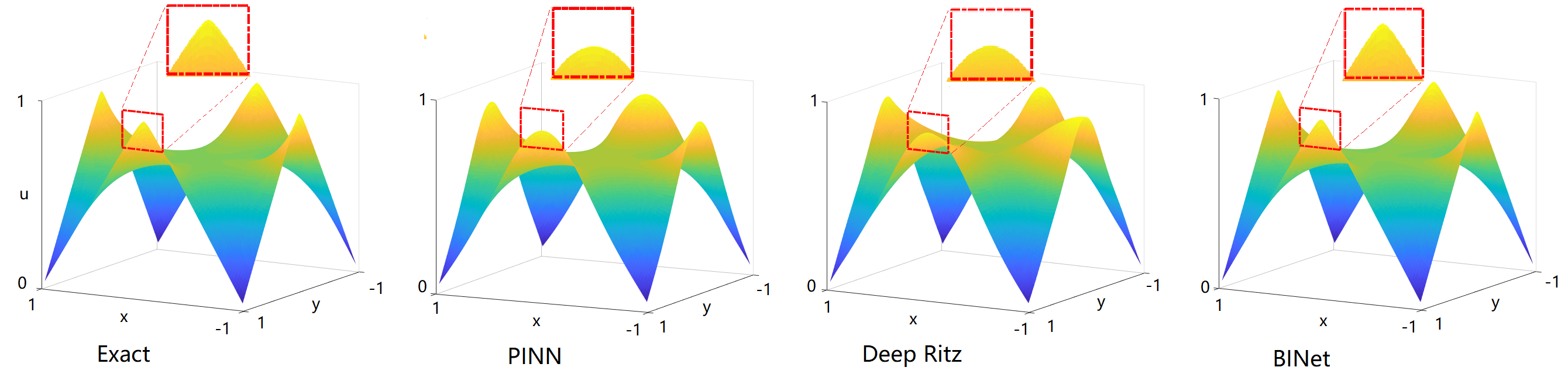}
		\centering
		\caption{The solutions of Laplace equation \eqref{Ex2_eq} with non-smooth boundary conditions by high-precision finite difference as the ground truth, PINN, deep Ritz and BINet. The red box is the zoom-in of the vicinity of the non-smooth point (0,1) in each figure with the same scale. Only BINet learns the singularity on the boundary successfully.}
		\label{f2}
	\end{figure}

	From Figure \ref{f2}, we can find that for PINN and Deep Ritz methods, the solutions on the boundary are smooth, which are different from the boundary condition. However, the boundary condition is well approximated by the solution of the BINet method. The reason is that, to satisfy the interior smoothness of the solution, the neural network of the PINN and Deep Ritz methods have to be a smooth function. So the solutions are still smooth even if they are close enough to the unsmooth boundary points.

	\noindent \textbf{Helmholtz Equation with Different Wavenumbers.}
	In this experiment, we consider an interior Helmholtz equation
	\begin{equation}
		\begin{aligned}
			-\Delta u(x,y) - k^2u(x,y) = 0,\ (x,y)\in\Omega,\\
			u(x,y) = e^{i(k_1x+k_2y)}, \ (x,y)\in\partial\Omega,
		\end{aligned}
		\label{eq:Ex3}
	\end{equation}
	where $(k_1,k_2)=(k\cos \frac{\pi}{7},k\sin\frac{\pi}{7})$, and $\partial\Omega=\{(\frac{9}{20}\cos(t)-\frac{\cos(5t)}{9}\cos(t),\frac{9}{20}\sin(t)-\frac{\cos(5t)}{9}\sin(t))|t\in[0,2\pi]\}$. The Deep-Ritz method can not solve the Helmholtz equation. Hence, we will compare the BINet method and PINN method for different $k$. We choose a fully connected neural network with 4 hidden layers with Sigmoid activation function and 40 neurons per layer.  and we choose 800 points on the boundary for BINet and PINN. In addition, we also randomly selected 2400 sample points in $\Omega$. For k = 1 and 4, we use the PINN type method and BINet method to solve the equation respectively. The loss function and results are shown in Figure \ref{f:Ex3}. We can see the loss function of BINet descends faster, and for $k=4$, the loss of the PINN method does not converge. In contrast, the loss of the BINet is always convergent no matter the value of $k$ is. The second and the third figures also show the result of BINet is much better than PINN. 
	\begin{figure}[t]
		\centering
		\includegraphics[width=0.3\textwidth]{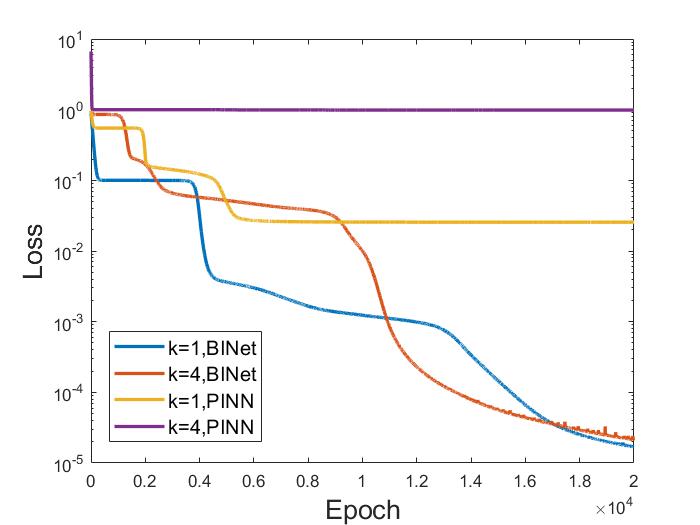}
		\includegraphics[width=0.3\textwidth]{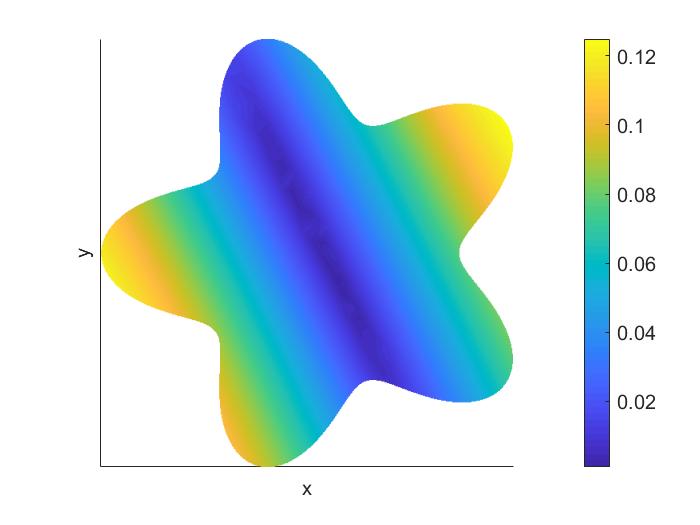}
		\includegraphics[width=0.3\textwidth]{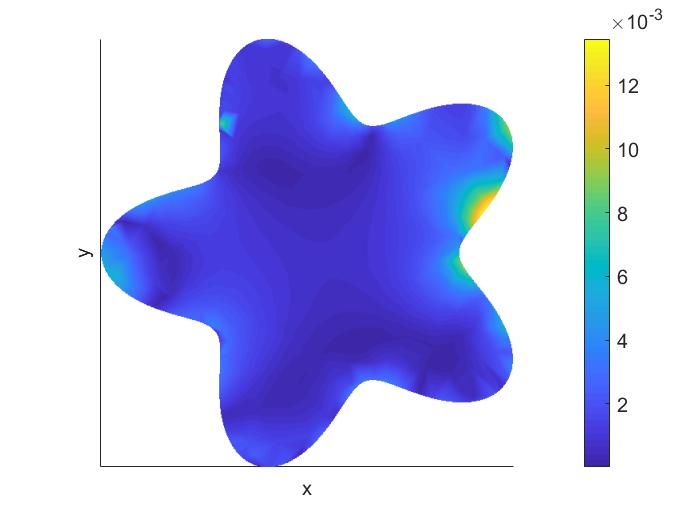}
		\caption{The experiments of solving the Helmholtz equation \eqref{eq:Ex3} by PINN and BINet. Deep Ritz is unable to solve the Helmholtz equation thus not exhibited. The first figure shows loss functions of different cases. Since for $k\geq 4$, PINN also fails, we show the errors when $k=1$. The errors of solutions solved by PINN method and BINet method are shown in the second and third figures, respectively.}
		\label{f:Ex3}
	\end{figure}

	
	\subsection{Experimental Results on Solution Operators}
	\textbf{The Operator from Equation Parameters to Solutions.}
	In this example, we consider the Helmholtz equations with variable wavenumber $k$.
	\begin{equation}
		\begin{aligned}
			-\Delta u(x,y) - k^2u(x,y) = 0,\ (x,y)\in\Omega,\\
			u(x,y) = H_0^1(k\sqrt{x^2+y^2}), \ (x,y)\in\partial\Omega,
		\end{aligned}
		\label{eq:Ex4}
	\end{equation}
	In the training phase, we set $k\in[2,3.5]\cup[4.5,6]$. We also use double layer potential to construct the loss function, and after 5000 training epochs,  we show the relative $L^2$ error versus the wavenumber $k$ in Figure \ref{fig:Ex4}. From the first figure, the relative $L^2$ error is about $10^{-3}$ or $10^{-2}$. Compared with solving a single equation, the relative $L^2$ error is still small. The relative error increases slightly with the increase of $k$, which is because the Helmholtz equation becomes more difficult to solve when $k$ increases. This means that we have successfully learned the operator mapping of exterior parametric PDE problems on an unbounded domain. Most importantly, although $k$ is not selected between $[3.5,4.5]$ during training, the relative error is still small on the test when we take values in the interval $[3.5,4.5]$. This shows that our method has good generalization ability.
	\noindent
	\begin{figure}[t]
		\centering
		\includegraphics[width=0.9\textwidth]{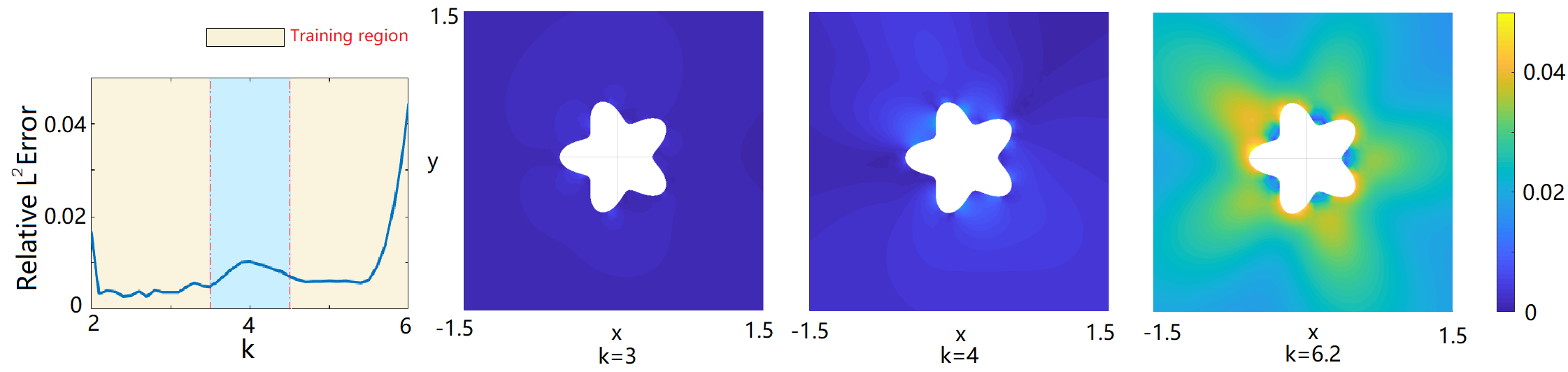}
		\caption{The results of the learning operators mapping the equation parameters $k$ (wavenumbers) to the solutions of Helmholtz equations \eqref{eq:Ex4}. The first figure shows the relative $L^2$ error of the solution with different wavenumber $k$, which shows BINet has successfully learned the solution operator and has generalization capability. The last three figures show the absolute error of the solutions mapping from different wavenumber $k$, i.e., $k=3, 4, 6.2$, respectively.}
		\label{fig:Ex4}
	\end{figure}

	\noindent \textbf{The Operator from Boundary Geometry to Solutions.}
	In this example, we consider a Laplace equation with parametric boundaries. The problem is 
	\begin{equation}
		\begin{aligned}
			-\Delta u(x,y) & = 0, \ &(x,y)\in\Omega_\beta,\\
			u(x,y) &= g(x,y;\beta), \ &(x,y)\in\partial\Omega_\beta,
		\end{aligned}
		\label{eq:ex5}
	\end{equation}
	where the boundary condition $g(x,y;\beta)=(x-x_{bc})(y-y_{bc})+(x-x_{bc})+(y-y_{bc})+1$, and $(x_{bc},y_{bc})$ is the barycenter of the $\Omega_\beta$. We assume that $\Omega_\beta$ can take any triangle in a domain. For simplicity, We can fix one vertex at the origin and one edge on the positive half x-axis, while the third vertex is in the first quadrant by translation and rotation. Then we can assume the vertex is $(0,0),(a,0), \text{and} (b,c)$. In this example, we assume $a,b,c$ can take any value in interval $[0,1]$. In this example, we choose a ResNet with eight blocks, and 100 neurons per layer. Single potential layer is used to calculate the boundary integral. We randomly selected 80 triangles to calculate the loss function, and after every 500 epochs, triangles will be randomly selected again. After training for 5000 epochs, we randomly choose two triangles, and the solutions of the each triangle by BINet method has shown in figure \ref{fig.Ex6}. The relative $L^2$ error is about $10^{-3}$. From this, we can see BINet has successfully learned the operator from boundary geometry to solution.
	\begin{figure}[t]
		\centering
		\includegraphics[width=0.9\textwidth]{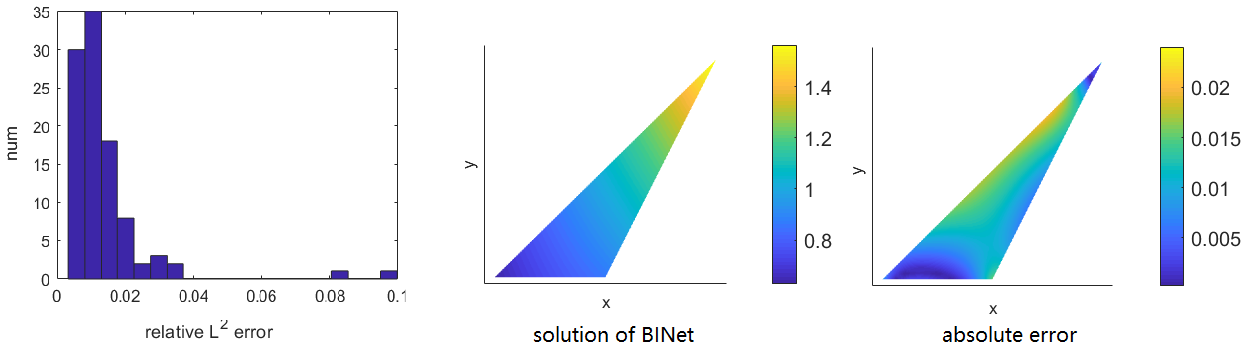}
		\caption{The results of learning operators from boundary geometry to solutions of the equation \eqref{eq:ex5}. This histogram of relative $L^2$ error of the solutions mapping from different triangle boundaries are shown in the first figure. One hundred triangle boundaries are randomly selected to give the distribution of the relative error. It is easy to check that more than half of the errors are less than $1\%$. For a typical triangle domain given by vertices $(0,0)$, $(0.3,0)$, $(0.6,0.4)$, the solution and its absolute error are shown in the right two figures respectively.}
		\label{fig.Ex6}
	\end{figure}
	
	\section{Conclusion}
	\label{sec:con}
	We have developed a new neural network method called BINet to solve PDEs. In BINet, the solution of PDE is represented by boundary integral composed of an explicit kernel and an unknown density which is approximated by a neural network. Then the PDE is solved by learning the boundary integral representation to fit the boundary condition. Since the loss function measures only the misfit between the integral representation and the boundary condition, BINet has less hyper-parameters and lower sampling dimensions than many other neural network-based PDE solvers. Because the boundary integral satisfies PDE automatically in the interior and exterior of the boundary,  BINet can solve bounded and unbounded PDEs. Furthermore, BINet can learn operators from PDE parameters including coefficients and boundary geometry to solutions. Besides, using the NTK technique, we prove that BINet converges as the width of the network goes to infinity.  We test BINet with the Laplace equation and Helmholtz equation in extensive settings. The numerical experiments show that BINet works effectively for many cases such as interior problems, exterior problems, high wavenumber problems. The experiments also illustrate the capability of BINet in learning solution operators. All the experiments verify the advantages of BINet numerically.  Although our method exhibits competitive performance against the PINN method and DeepRitz method in many situations, the requirement of high-precision boundary integration limits further applications in higher-dimensional problems. This will be the direction of improving BINet in the future.


	\bibliographystyle{unsrt} 
	\bibliography{refs}
	
	\newpage

	\appendix
	\section*{Appendix}
	\section{A review of the PINN and Deep Ritz method}
	\label{ap1}
	\subsection{PINN method}
	To solve the linear PDE
	\begin{equation}
		\begin{aligned}
			\mathcal{L}u(x)=0, \ x\in\Omega,\\
			u(x) = g(x), \ x\in\partial\Omega,
		\end{aligned}
		\label{eq:app1}
	\end{equation}
	the main idea of the PINN\cite{raissi2019physics} method is to use a neural network $u(x;\theta)$ as an ansatz to approximate the solution $u(x)$, where $\theta$ represents the trainable parameters in the neural network. There was other work of the similar idea such as \cite{berg2018unified,lagaris2000neural,sirignano2018dgm}. Then we can use the automatic differentiation tool to calculate the derivative $\mathcal{L}u(x;\theta)$ and define the loss function 
	$$L_1(\theta) = \|\mathcal{L}u(x;\theta)\|^2_{\Omega}.$$ For the boundary conditions, we can define the loss function 
	$$L_2(\theta)=\|u(x;\theta)-g(x)\|^2_{\partial\Omega}.$$
	Finally, we can combine the loss function $L_1$ and $L_2$ with a hyper-parameter $\beta$ to get loss function,
	$$
	L(\theta) = L_1+\beta L_2 = \|\mathcal{L}u(x;\theta)\|^2_{\Omega}+\beta\|u(x;\theta)-g(x)\|^2_{\partial\Omega}.
	$$
	By minimizing the loss function $L$, PINN will get the approximation solution of the PDE \eqref{eq:app1}.
	
	\subsection{Deep Ritz method}
	For the specific PDE problems in equation \eqref{eq:app1}, we can change the equation into a Ritz variational form. This is the main idea of the Deep Ritz method\cite{weinan2018deep}. For instance, if we consider a Laplace equation,
	\begin{equation}
		\begin{aligned}
			-\Delta u(x)=0,\ x\in\Omega,\\
			u(x) = g(x), \ x\in\partial\Omega,
		\end{aligned}
		\label{eq:app2}
	\end{equation}
	we can solve the equation equivalently by minimizing the following Ritz variational problem
	\begin{equation}
		\begin{aligned}
			\int_{\Omega}\frac{1}{2}|\nabla u(x) |^2 dx.
		\end{aligned}
		\label{eq:Ritz}
	\end{equation}
	We also use a neural network $u(x;\theta)$ to approximate the solution of the PDE, and we can use the automatic differentiation tool to calculate the gradient $\nabla u(x;\theta)$ of the neural network. So the variation \eqref{eq:Ritz} can naturally be used as a loss function, defined as  
	$$
	L_1(\theta) = \int_{\Omega}\frac{1}{2}|\nabla u(x;\theta) |^2 dx.
	$$
	For the boundary condition, the loss function $L_2$ also can be defined as 
	$$L_2(\theta)=\|u(x;\theta)-g(x)\|^2_{\partial\Omega}.$$
	Finally, the loss funtion can be defined as \begin{equation}
		L(\theta) = L_1+\beta L_2 =\int_\Omega\frac{1}{2}|\nabla u(x;\theta)|^2dx+\beta\|u(x;\theta)-g(x)\|^2_{\partial\Omega},
	\end{equation}  
	where $\beta$ is also the hyper-parameter. By minimize the loss function $L(\theta)$, Deep Ritz method will get the solution of the PDE \eqref{eq:app2}.
	
	This paper introduce a residual network as the anastz to approximate the solution. The residual network is also used in our work to approximate the density function in the boundary integral form. The architecture of the residual network is shown in Figure \ref{fig.ritz}.
	
	\begin{figure}[t]
		\centering
		\centering
		\includegraphics[width=0.8\textwidth]{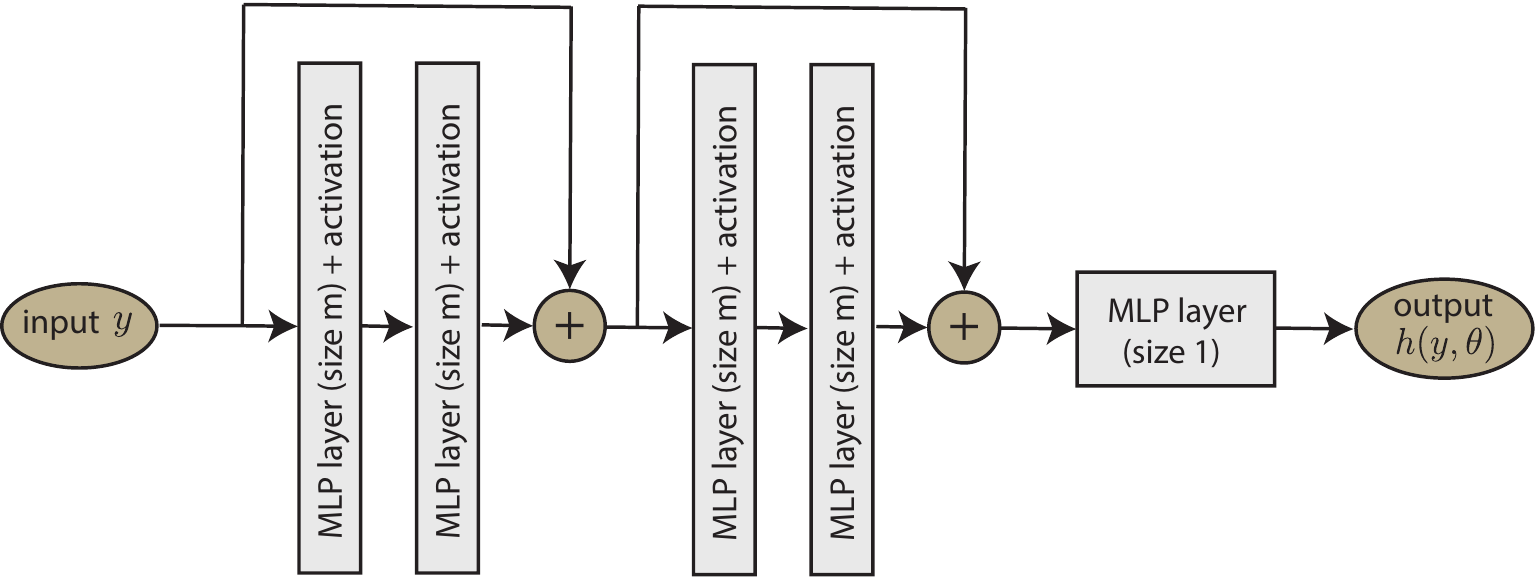}
		\centering
		\caption{The architecture of the residual network.}
		\label{fig.ritz}
	\end{figure}

	\section{The fundamental solution of different equations}
	In this section, we make some supplementary introductions to the fundamental solution. Before defining the fundamental solution, we first introduce the $\delta$-function, 
	\begin{definition}
		A function $\delta(x)$ called n-dimensional $\delta$-function if $\delta(x) = \left\{\begin{array}{cc}
			0, & x \neq 0 \\
			\infty, & x=0,
		\end{array}\right.$, and for all functions $f$ that is continuous at $a$ we have
		$\int_{\mathbb{R}^n} f(x)  \delta(x-a) d x=f(a).$
	\end{definition}
	Then, for the PDE
	\begin{equation}
		\mathcal{L}u(x)=0,\ x\in\Omega,
		\label{eq:app0}
	\end{equation}
	we can define the corresponding fundamental solution $G$ of the equations (\ref{eq:app0}).
	\begin{definition}
		A function $G(x,y)$ is called the fundamental solution corresponding to equations (\ref{eq:app0}) if $G(x,y)$ is symmetric about $x$ and $y$ and $G(x,y)$ satisfy
		$$
		\mathcal{L}_yG(x,y)=\delta(|y-x|),
		$$
		where $(x,y)\in\mathbb{R}^n\times\mathbb{R}^n$ and $L_y$ is the differential operator $L$ which acts on component $y$.
	\end{definition}

	Limited by the length of the article, although we only introduce the fundamental solutions of Laplace equations and Helmholtz equations in two-dimensional cases in detail, in general, BINet can be applied as long as the fundamental solution of $\mathcal{L}$ in $\mathbb{R}^d$ can be obtained. Let's give a few more examples. More details can be found in \cite{hsiao2008boundary,kellogg1953foundations,ying2009fast}.
	\subsection{The Laplace Equations}
	If we consider a Laplace equation
	\begin{equation}
		-\Delta u(x)=0, \ x\in\Omega,
	\end{equation}
	the fundamental solution has the following form 
	
	$$
	G(x,y)=\left\{\begin{array}{cc}
		-\frac{1}{2\pi}ln|x-y|  & n=2 \\[0.5em]
		\frac{1}{(n-2)w_n}\frac{1}{|x-y|^{n-2}} & n\geq 3, 
	\end{array}
	\right.
	$$
	where $n$ is the dimension of the equation and $w_n$ is the volume of the n-dimensional unit sphere. Then the fundamental solution $G$ satisfies
	$$
	-\Delta_yG(x,y) = \delta(|x-y|).
	$$

	\subsection{The Helmholtz Equations}
	The Helmholtz equation has the following form
	\begin{equation}
		-\Delta u(x)-k^2u(x)=0, \ x\in\Omega,
	\end{equation}
	where $k$ is a real number. The fundamental solution of the Helmholtz equation has the following form 
	$$
	G(x,y)=\left\{\begin{array}{cc}
		\frac{i}{4}H_0^1(k|x-y|)  & n=2 \\[0.2em]
		\frac{1}{(n-2)w_n}\frac{e^{ik|x-y|}}{|x-y|^{n-2}} & n\geq 3, 
	\end{array}
	\right.
	$$
	where $n$ is the dimension of the equation and $w_n$ is also the volume of the n-dimensional unit sphere. Then the fundamental solution $G$ satisfies
	$$
	-\Delta_yG(x,y)-k^2G(x,y) = \delta(|x-y|).
	$$

	\subsection{The Navier's Equations}
	We consider Navier's equations (also called Lam$\acute{e}$ system). These are famous equations in linear elasticity for isotropic materials, and the governing equations are
	\begin{equation}
		\begin{aligned}
			-\mu\Delta u(x) - (\lambda+\mu)\nabla(\nabla\cdot u(x)) = 0,
		\end{aligned}
		\label{Navier}
	\end{equation}
	where $\lambda,\mu>0$ are the Lam$\acute{e}$ constants of the elastic material, and $u(x)\in\mathbb{R}^n$ is the displacement vector. The fundamental solution of the equation \eqref{Navier} is 
	\begin{equation}
		G(x, y)=\left\{\begin{aligned}
			\frac{\lambda+3 \mu}{8 \pi \mu(\lambda+2 \mu)}\left[\frac{1}{|x-y|} \mathbf{I}_{3}\right.\\
			\left.+\frac{\lambda+\mu}{\lambda+3 \mu} \frac{1}{|x-y|^{2}}(x-y)(x-y)^{\mathrm{T}}\right], & n=3, \\
			\frac{\lambda+3 \mu}{4 \pi \mu(\lambda+2 \mu)}\left[\ln \frac{1}{|x-y|} \mathbf{I}_{2}\right.\\
			\left.+\frac{\lambda+\mu}{\lambda+3 \mu} \frac{1}{|x-y|^{2}}(x-y)(x-y)^{\mathrm{T}}\right], & n=2.
		\end{aligned}\right.
		\label{FS_navier}
	\end{equation}
	It means that G(x,y) defined by the \eqref{FS_navier} satisfies the following equation,
	\begin{equation}
		-\mu\Delta_y G(x,y) - (\lambda+\mu)\nabla_y(\nabla_y\cdot G(x,y)) = \delta(|x-y|)\mathbf{I}_n, 
	\end{equation}
	where $\mathbf{I}_n$ is the n-order identity matrix.
	
	\subsection{The Stokes Equations}
	Stokes equations are well known in the incompressible viscous fluid model. The general form of the Stokes equations is 
	\begin{equation}
		\begin{aligned}
			-\mu\Delta u(x)+\nabla p(x)&=f(x),\\
			\nabla\cdot u(x)&=0,\ x\in\Omega\subset\mathbb{R}^n,
		\end{aligned}
		\label{eq:stoke}
	\end{equation}
	where $u$ and $p$ are the velocity and pressure of the fluid flow, respectively, and $\mu$ and $f$ are the given dynamic viscosity of the fluid and forcing term, respectively. 
	
	For n=2 the fundamental solutions of the \eqref{eq:stoke} are
	\begin{equation}\begin{aligned}
			\boldsymbol{v}^{k}(x, y) &=\frac{1}{4 \pi \mu}\left\{\log \frac{1}{|x-y|} \boldsymbol{e}^{k}+\sum_{j=1}^{2} \frac{\left(x_{k}-y_{k}\right)\left(x_{j}-y_{j}\right) \boldsymbol{e}^{j}}{|x-y|^{2}}\right\} \\
			q^{k}(x, y) &=\frac{\partial}{\partial x_{k}}\left\{-\frac{1}{2 \pi} \log \frac{1}{|x-y|}\right\},
	\end{aligned}\end{equation}
	and for n=3 the fundamental solutions of the \eqref{eq:stoke} are
	\begin{equation}\begin{aligned}
			\boldsymbol{v}^{k}(x, y) &=\frac{1}{8 \pi \mu}\left\{\frac{1}{|x-y|} \boldsymbol{e}^{k}+\sum_{j=1}^{3} \frac{\left(x_{k}-y_{k}\right)\left(x_{j}-y_{j}\right) \boldsymbol{e}^{j}}{|x-y|^{3}}\right\} \\
			q^{k}(x, y) &=\frac{\partial}{\partial x_{k}}\left\{-\frac{1}{4 \pi} \frac{1}{|x-y|}\right\},
	\end{aligned}\end{equation}
	where $k=1,\cdots,n$ and $e^k$ denotes the unit vector along the $x_k$-axis. $v^k$ and $q^k$ satisfy 
	\begin{equation}
		\begin{aligned}
			-\mu\Delta_x v^k(x,y)+\nabla_x q^k(x,y)&=\delta(|x-y|)e^k,\\
			\nabla_x\cdot v^k(x,y)&=0,
		\end{aligned}
		\label{eq:stoke}
	\end{equation}
	where $x,y\in\mathbb{R}^n$.
	
	\subsection{The Biharmonic Equation}
	The Biharmonic Equation is a single scalar 4th-order equation, which can be reduced from plane elasticity and plane Stokes flow. We consider a two dimensional Biharmonic equation,
	\begin{equation}
		\Delta^2u(x)=0, \ x\in\Omega\subset\mathbb{R}^2.
		\label{Bihar}
	\end{equation}
	The fundamental solution of the equation \eqref{Bihar} is 
	\begin{equation}
		\begin{aligned}
			G(x,y) = \frac{1}{8\pi}|x-y|^2\log|x-y|, \ x,y\in\mathbb{R}^2,
		\end{aligned}
	\end{equation}
	where $G(x,y)$ satisfies 
	$$\Delta_y^2 G(x,y)=\delta(|x-y|).$$

	\section{The convergence analysis of BINet}
	
	\subsection{The structure for solving PDEs using neural networks}\label{ap.pdeneural}
	BINet consists of a neural network such as MLP and an integral operator performed on the output of the neural network. 
	Thus, the output of BINet reads $$v(x,\theta)=\mathcal{A}[h](x,\theta),$$ 
	where $h=h(y,\theta)\in\mathbb{R}$ is the neural network approximating the density function in the boundary integral form, $y \in\mathbb{R}^d$ is the $d$-dimensional variable. The operator $\mathcal{A}$ is performed on the output of the neural network which completes the whole architecture.  And the loss function is  $$\mathcal{L} = \|v(x,\theta)-\tilde{v}(x)\|^2_2,$$ with label function $\tilde{v}(x)$.
	
	For a more general setup, the operator $\mathcal{A}$ has different forms. 
	For PINN/DGM method, the operator is directly the partial differential operator, implying
	\begin{equation*} 
		\mathcal{A}[u](x,\theta) = -\Delta u(x,\theta),
	\end{equation*}
	where $u(x,\theta)$ is the approximation of the solution.
	The Deep-Ritz method for solving the Laplace equation $-\Delta u(x) = 0$ is to minimize the optimization problem $\min_{\theta} L(\theta)$  where part of the loss reads
	\begin{align*}
		L_1(\theta)&=\int_{\Omega}\frac{1}{2}|\nabla u(x,\theta)|^2dx.
	\end{align*}
	It follows that the corresponding operator has the following form
	\begin{equation*}
		\mathcal{A}[u](x) = \frac{1}{\sqrt{2}}\nabla u(x,\theta).
	\end{equation*}
	Therefore in the view of the operator applied on the neural network, different from the integral type operator of BINet, PINN and Deep Ritz methods have extra differential operators although the Deep Ritz method decreases the order from the second to the first.

	\begin{definition}\label{def.ad}
		The operator $\mathcal{A}$ is admissible if the following conditions hold: \begin{enumerate}
			\item $\mathcal{A}(\lambda_1 f_1+\lambda_2f_2)=\lambda_1 \mathcal{A}f_1+\lambda_2\mathcal{A}f_2$ (linear property);
			\item $[\mathcal{A}_x, \mathcal{A}_{x'}]:=\mathcal{A}_x\mathcal{A}_{x'}-\mathcal{A}_{x'}\mathcal{A}_x=0$ (commutative property);
			\item $[\mathcal{A},\frac{\pl}{\pl \theta}]=0$ (parameter variant).
		\end{enumerate}
	\end{definition}
	It is easy to check that the operators of PINN, Deep-Ritz, and our BINet all satisfy the admissible property. And the admissibility is crucial in the following proof.
	
	The different design of the neural network and the operator makes the network different. Here, we adopt the typical settings of the neural network as an MLP. As the integral operator $\mathcal{A}$ is bounded, thus the convergence results can be obtained in our BINet and the proof is shown in Appendix 3.3. Here the structure of the neural network is introduced first for the derivation of the NTK form.
	
	The $L$-hidden layer MLP is defined as 
	\begin{align}
		\text{input layer: } & \quad g^{(0)} = y,\\
		\text{hidden layer:} & \quad g^{(l)}(y) \overset{\Delta}{=} \sqrt{\frac{c_\sigma}{d_l}}\sigma(f^{(l)}(y)), \;\; f^{(l)}(y)= W^{(l)}g^{(l-1)}(y), \; l\in[L]\\
		\text{output layer:} & \quad h(y,\theta)=f^{(L+1)}=W^{(L+1)}g^{(L)}(y),
	\end{align}
	where $l\in[L]\overset{\Delta}{=}\{1,2,\cdots,L\}$ is the hidden layer, $\theta$ is the trainable parameters which is the standard representation for the weights $W^{(l)}$, $d_l$ is the width of the $l$-th layer and $c_\sigma={\bf E}_{u\sim N(0,1)}\sigma(u)$. 
	


	\subsection{The dynamic neural tangent kernel}
	
	We have chosen the MLP as the neural network in the analysis for simplicity. A similar analysis can also be done for other structures as the convergence results of such neural networks are reported in the literature \cite{allen2018learning, huang2020dynamics}. Applying the integral operator on the neural networks should also give similar convergence results. Thus different schemes here imply different forms of the operator $\mathcal{A}$, see Appendix \ref{ap.pdeneural} for detail.

	The training process of the neural ODE is basically to minimize the loss by the method based on the gradient. One typical scheme is the gradient descent method which has the form 
	\begin{equation}
		\theta_{n+1} = \theta_n -\alpha\frac{\pl L}{\pl \theta}.
	\end{equation}
	When the learning rate $\alpha\to0$, we have the limiting gradient flow
	\begin{equation}
		\frac{d\theta}{d t} = -\frac{\pl L}{\pl \theta},
	\end{equation}
	where $t$ is the continuous version of index of the learning steps in the training process. More precisely, for the weight matrix $W^{(l)}$, we have the evolution 
	$\frac{dW^{(l)}}{d t} = -\frac{\pl L}{\pl W^{(l)}}.$
	
	Hence the evolution of the prediction satisfies the following form
	\begin{align}
		\begin{split}
			\frac{dv(x,\theta)}{dt} &= \sum_{\theta_p}\frac{\pl v}{\pl \theta_p}\frac{\pl \theta_p}{\pl t} \\
			&=\sum_{\theta_p}-\mathcal{A}[\frac{\pl h}{\pl \theta_p}](x)\frac{\pl L}{\pl \theta_p}\\
			&=\sum_{\theta_p}-\mathcal{A}[\frac{\pl h}{\pl \theta_p}](x)\langle\frac{\pl L}{\pl v}(x'),\frac{\pl v}{\pl \theta_p}(x')\rangle\\
			&=\sum_{\theta_p}-\mathcal{A}[\frac{\pl h}{\pl \theta_p}](x)\langle \zeta(x'),\mathcal{A}[\frac{\pl h}{\pl \theta_p}](x')\rangle\\
			&=-\langle \zeta(x'), \sum_{\theta_p}\mathcal{A}[\frac{\pl h}{\pl \theta_p}](x)\mathcal{A}[\frac{\pl h}{\pl \theta_p}](x')\rangle,
		\end{split}
	\end{align}
	where $\zeta(x') = \frac{\pl L}{\pl v}(x')=v(x')-\tilde{v}(x')$ is the vector of loss,  and $p$ denotes the index of the learning parameter. We denote the dynamic Neural Tangent Kernel (DNTK) for the PDE-based neural network as 
	\begin{equation}
		\mathcal{N}(x,x') = \sum_{\theta_p}\mathcal{A}[\frac{\pl f}{\pl \theta_p}](x)\mathcal{A}[\frac{\pl f}{\pl \theta_p}](x')=\sum_{l}\langle\mathcal{A}[\frac{\pl f}{\pl W^{(l)}}](x),\mathcal{A}[\frac{\pl f}{\pl W^{(l)}}](x')\rangle_W,
	\end{equation}
	where $\langle\cdot,\cdot\rangle_W$ is defined as the summation over each component index of $W^{(l)}$.
	
	Next, we would give the explicit form of the DNTK for further analysis.
	Recall that the output of the MLP in PDE-based neural network has the following form 
	\begin{align}
		\begin{split}
			h(x,\theta) &= f^{(L+1)}(x) = W^{(L+1)}g^{(L)}(x)\\
			& = W^{(L+1)}\sqrt{\frac{c_\sigma}{d_L}}\sigma\biggl(\cdots W^{(l+1)}\sqrt{\frac{c_\sigma}{d_l}}\sigma\bigl(W^{(l)}\sqrt{\frac{c_\sigma}{d_{l-1}}}\sigma(\cdots) \bigr)\cdots\biggr),
		\end{split}
	\end{align}
	where we have omitted the explicit dependence of $\theta$ in the formula for simplicity. And thus
	\begin{equation}
		h(x,\theta)=W^{(L+1)}\sqrt{\frac{c_\sigma}{d_L}}\sigma\biggl(\cdots f^{(l+1)} \cdots\biggr),
	\end{equation}
	where 
	\begin{align}
		\begin{split}
			f^{(l+1)} &= W^{(l+1)}g^{(l)},\\
			g^{(l)} &= \sqrt{\frac{c_\sigma}{d_l}}\sigma\bigl(f^{(l)}\bigr).
		\end{split}
	\end{align}
	To give the form of the DNTK, the key is to give the form of $\frac{\pl h}{\pl W^{(l)}}$. From above forms, we can obtain
	\begin{align}
		\begin{split}
			\frac{\pl h}{\pl W^{(l)}} & = \frac{\pl h}{\pl f^{(l)}}\frac{\pl f^{(l)}}{\pl W^{(l)}}\\
			& = \frac{\pl h}{\pl f^{(l+1)}}\frac{\pl f^{(l+1)}}{\pl g^{(l)}}\frac{\pl g^{(l)}}{\pl f^{(l)}}\frac{\pl f^{(l)}}{\pl W^{(l)}}\\
			& = \biggl(\frac{\pl h}{\pl f^{(l+1)}}W^{(l+1)}\sqrt{\frac{c_\sigma}{d_l}}S^{(l)}\biggr)^T\bigl(g^{(l-1)}\bigr)^T\\
			& = \sqrt{\frac{c_\sigma}{d_l}}S^{(l)}\bigl(W^{(l+1)}\bigr)^T\bigl(\frac{\pl h}{\pl f^{(l+1)}}\bigr)^T\bigl(g^{(l-1)}\bigr)^T,
		\end{split}
	\end{align}
	where we have used the denotation
	\begin{equation}
		\bigl[S^{(l)}\bigr]_{ij} = \bigl[\dot{\sigma}(f^{(l)})\bigr]_i\delta_{ij}.
	\end{equation}
	
	By defining $b^{(l)} = \bigl(\frac{\pl h}{\pl f^{(l)}}\bigr)^T$ and we obtain
	\begin{equation}
		\frac{\pl h}{\pl W^{(l)}} = b^{(l)}\bigl(g^{(l-1)}\bigr)^T, l\in[L]
	\end{equation}
	where $b^{(l)}$ satisfies the induction relation
	$b^{(l)} = \sqrt{\frac{c_\sigma}{d_l}}S^{(l)}\bigl(W^{(l+1)}\bigr)^Tb^{(l+1)}$.
	
	With the admissible property of $\mathcal{A}$ in the sense of Definition \ref{def.ad}, we have the DNTK as
	\begin{align}
		\begin{split}
			\mathcal{N}(x,x') &=\sum_{l}\bigl\langle\mathcal{A}[\frac{\pl h}{\pl W^{(l)}}](x),\mathcal{A}[\frac{\pl h}{\pl W^{(l)}}](x')\bigr\rangle_W\\
			& = \sum_l\bigl\langle\mathcal{A}[b^{(l)}\bigl(g^{(l-1)}\bigr)^T](x) ,\mathcal{A}[b^{(l)}\bigl(g^{(l-1)}\bigr)^T](x') \bigl\rangle\\
			& = [\mathcal{A}\mathcal{K}\mathcal{A}](x,x').
		\end{split}
	\end{align}
	where $\mathcal{K}(y, y')$ is the dynamic neural tangent kernel of the MLP with the following form \cite{allen2019convergence}
	\begin{equation}
		\mathcal{K}(y,y')=\sum_l\bigl\langle b^{(l)}(y)\bigl(g^{(l-1)}(y)\bigr)^T, b^{(l)}(y')\bigl(g^{(l-1)}(y')\bigr)^T \bigl\rangle,
	\end{equation} and $[\mathcal{A}\mathcal{K}\mathcal{A}](x,x')$ is a function given by the kernel $\mathcal{K}(y,y')$ operated by $\mathcal{A}$ on its head and the tail for performing with respect to the former variable $y$ and latter variable $y'$. 
	
	Denote the constant neural tangent kernel in \cite{arora2019exact} as 
	\begin{equation}
		\Theta^{(L)}(y,y')=\sum_{l=1}^{L+1}\biggl(\Sigma^{(l-1)}(y,y')\prod_{l'=l}^{L+1}\dot{\Sigma}^{(l')}(y,y')\biggr),
	\end{equation}
	where $\Theta^{(L)}(y,y')$ is given by a reduction form
	\begin{align}
		\Sigma^{(0)}(y,y') &= y^Ty',\\
		\Lambda^{(l)}(y,y') &= \begin{pmatrix}
			\Sigma^{(l-1)}(y,y) & \Sigma^{(l-1)}(y,y')\\
			\Sigma^{(l-1)}(y',y) & \Sigma^{(l-1)}(y',y')\\ 
		\end{pmatrix},\\
		\Sigma^{(l)}(y,y') &= c_\sigma\mathbb{E}_{(u,v)\sim \mathcal{N}(0,\Lambda^{(l)})}[\sigma(u)\sigma(v)],\\
		\dot{\Sigma}^{(l)}(y,y') &= c_\sigma\mathbb{E}_{(u,v)\sim \mathcal{N}(0,\Lambda^{(l)})}[\dot{\sigma}(u)\dot{\sigma}(v)],\\
	\end{align}
	and $\dot{\Sigma}^{(L+1)}(y,y')=1$.


	The convergence results of both initialization and during training depend on the Gaussian random initialization. The parameters $W^{(l)},l=0,1,\cdots,L+1$ are initialized by random Gaussian, i.e., 
	\begin{equation}
		W^{(l)}\sim \mathcal{G}(0,I),
	\end{equation}
	where $\mathcal{G}$ is the Gaussian distribution with the identity covariance matrix $I$. Such initialization agrees with the so-called ``LeCun"\cite{lecun2012efficient} and``Kaiming"\cite{he2015delving} initialization with only a constant difference.

	The convergence results of the initialization case and during-training case are summarized in Theorem 4.1 and 4.2 respectively. And the main proof is given in the Appendix \ref{ap.proof} using the results in \cite{arora2019exact}  and the bounded properties of the operator $\mathcal{A}$.

	\subsection{Proof of Lemma 1}\label{ap.definite}
	\begin{proof}
		The result can be obtained using the positive definiteness of the kernel $\Theta^{(L)}$ and the invertibility of the operator $\mathcal{A}$. 
		For any  $f\neq 0\in L^{2}(\partial\Omega)$, we have
		\begin{align*}
			&\int\int f(x) [\mathcal{A}\theta^{(L)}\mathcal{A}](x,x') f(x')dxdx' \\
			= & \int\int f(x) \int\int \tilde{G}(x,y) \Theta^{(L)}(y,y')\tilde{G}(x',y') dydy' f(x')dxdx'\\
			= & \int\int dydy' \int f(x)  \tilde{G}(x,y)dx\, \Theta^{(L)}(y,y')  \int \tilde{G}(x',y') f(x')dx'\\
			= & \int\int l(y) \Theta^{(L)}(y,y')  l(y')dydy',\\
		\end{align*}
		where $l(y) = \int f(x)  \tilde{G}(x,y)dx$ and we have omitted the integral domain $\partial\Omega$ for simplicity.  For double layer potential, $\mathcal{A}: L^{2}(\partial\Omega)\to L^{2}(\partial\Omega)$ is invertible. \cite{verchota1984layer}  For single layer potential, there is an invertible theorem but with some constraints, i.e., $\mathcal{A}: L^{2}(\partial\Omega)\to L_1^{2}(\partial\Omega)$ is invertible if $\partial\Omega$ is a $C^\infty$ compact domain. \cite{gao1991layer} 
		Therefore $l\neq 0\in L^2(\partial\Omega)$  holds for double layer potential but holds for single layer potential with constraints.  As the constant kernel $\Theta^{(L)}(y,y')$ is positive definite proven in \cite{jacot2018neural}, the new kernel $[\mathcal{A}\Theta^{(L)}\mathcal{A}](x,x')$ is thus positive definite.
	\end{proof}

	\subsection{Proof of Theorem 4.1 and Theorem 4.2}
	\label{ap.proof}
	In this part, we would give proof of Theorem 4.1 and 4.2. The proof is based on the result of \cite{arora2019exact}.

	\begin{proof}(Proof of Theorem 4.1)
		
		\begin{lemma}
			(\text{Adopted from Theorem 3.1 from \cite{arora2019exact}})
			Fix $\epsilon>0$ and $\delta\in(0,1)$. Suppose the activation nonlinear function $\sigma(\cdot) = max\{\cdot,0\}$ is ReLU, the minimum width of the hidden layer $\text{min}_{l\in[L]}d_l\geq \Omega(\frac{L^6}{\epsilon^4}\log(L/\delta))$. Then for the normalized data $y$ and $y'$ where $\|y\|\leq1$ and $\|y'\|\leq1$, with probability at least $1-\delta$
			we have $$|\mathcal{K}_0(y,y')-\Theta^{(L)}(y,y')|\leq (L+1)\epsilon,$$
			where $\mathcal{K}_0(y,y')=\sum_l\bigl\langle b^{(l)}(y,0)\bigl(g^{(l-1)}(y,0)\bigr)^T, b^{(l)}(y',0)\bigl(g^{(l-1)}(y',0)\bigr)^T \bigl\rangle$ is the neural tangent kernel at initialization, i.e., $t=0$. 
			\label{lemma.a1}
		\end{lemma}
		
		By Lemma \ref{lemma.a1}, the result of Theorem 4.1 is directly obtained. 
	\end{proof}

	\begin{proof}(Proof of Theorem 4.2)
		
		Denote $W^{(l)}$ the learning parameter and $\tilde{W}^{(l)}$ the perturbation parameter with the perturbation matrices $\Delta W^{(l)}$, i.e., $\tilde{W}^{(l)}=W^{(l)}+\Delta W^{(l)}$, where $\|\Delta W^{(l)}\|_F$ is bounded. Presume that the input data satisfy a distribution $\mathcal{P}$ with the measure $P_{in}$. Let $\int P_{in}(dx) = 1$ and we denote $dx=P_{in}(dx)$ for simplicity.

		Let $\tilde{\mathcal{N}}(x,x')$ denote the perturbation of $\mathcal{N}(x,x')$ by the perturbation matrices $\Delta W^{(l)}$, i.e., 
		\begin{align}
			\mathcal{N}(x,x') & = \sum_l\langle\mathcal{A}[\frac{\partial f}{\partial W^{(l)}}](x), \mathcal{A}[\frac{\partial f}{\partial W^{(l)}}](x')\rangle_W\\
			\tilde{\mathcal{N}}(x,x') & = \sum_l\langle\mathcal{A}[\frac{\partial f}{\partial \tilde{W}^{(l)}}](x), \mathcal{A}[\frac{\partial f}{\partial \tilde{W}^{(l)}}](x')\rangle_W\\
		\end{align}
		
		\begin{lemma}(The reduction of the kernel of BINet)
			
			\begin{equation}\label{eq.reduce}
				|\mathcal{N}_t(x,x')-\mathcal{N}_0(x,x')| \leq A^2\|\mathcal{K}_t-\mathcal{K}_0\|_\infty,
			\end{equation}
			where $\|\mathcal{K}\|_\infty = \sup_{y,y'} |\mathcal{K}(y,y')|$.
			\label{lemma.reduction}
		\end{lemma}

		\begin{proof}
			\begin{align}
				\begin{split}
					|\tilde{\mathcal{N}}(x,x')-\mathcal{N}(x,x')| & = |\mathcal{A}\sum_l (\langle\frac{\partial f}{\partial \tilde{W}^{(l)}}, \frac{\partial f}{\partial \tilde{W}^{(l)}}\rangle-\langle\frac{\partial f}{\partial W^{(l)}}, \frac{\partial f}{\partial W^{(l)}}\rangle)\mathcal{A}|\\
					& \leq \|\mathcal{A}\|_\infty\|\mathcal{A}\|_\infty\|\widetilde{\mathcal{K}}-\mathcal{K}\|_\infty\\
					&\leq A^2\|\widetilde{\mathcal{K}}-\mathcal{K}\|_\infty.
				\end{split}
			\end{align}
			We have shown that the perturbation of the new kernel can be controlled by the perturbation of the kernel of the neural network. Thus, the perturbation during training of the new kernel of BINet can be controlled by that of the kernel of the neural network. We complete the proof.
		\end{proof}
		
		
		\begin{lemma}(Adopted from \cite{arora2019exact})
			We have for all $t\geq 0$, $$|\mathcal{K}_t(y,y')-\mathcal{K}_0(y,y')|\leq \omega,$$ if  the following holds 
			\begin{equation}
				\|W^{(l)}(t)-W^{(l)}(0)\|_F = \mathcal{O}(A\frac{\sqrt{n}}{\lambda_0})\leq \omega\sqrt{m}
			\end{equation}
			for any $l$ and $t$.
			Here $\mathcal{K}_t(y,y')$ is the kernel along training step $t$ and $\mathcal{K}_0(y,y')$ is the kernel.
			\label{lemma.kernel}
		\end{lemma}

		Lemma \ref{lemma.kernel} is a trivial property of the neural network from \cite{arora2019exact}. But the existence of its condition in BINet is nontrivial because of the operator $\mathcal{A}$. Luckily in our BINet, the condition can be proven to hold since our integral operator is bounded. Specifically, from \cite{arora2019exact} if the following lemma holds for our BINet, the condition of Lemma \ref{lemma.kernel} exists. Thus the only thing left is to verify if the following lemma still holds for our BINet.
		
		\begin{lemma}(Adopted from Lemma F.7 in \cite{arora2019exact})
			Fix $\omega\leq \text{poly}(1/L,1/n,1/\log(1/\delta),\lambda_0)$ and $\delta\in(0,1)$. Suppose that $\text{min}_l{d_l}\geq \text{poly}(1/\omega)$. Fixed $l'\in[L+1]$, we have with probability at least $1-\delta$ over random initialization, for all $t\geq 0$
			\begin{equation}
				\|W^{(l)}(t)-W^{(l)}(0)\|_F = \mathcal{O}(A\frac{\sqrt{n}}{\lambda_0})\leq \omega\sqrt{m},
			\end{equation}
			if following inequalities hold for $\forall l\in[L+1]\backslash \{l'\}$:
			\begin{align}
				\begin{split}
					&\|{\bf \tilde{v}}_{nn}(t)-{\bf v}\|_2\leq \exp(-\frac{1}{2}\kappa^2\lambda_0t)\|{\bf \tilde{v}}_{nn}(0)-{\bf v}\|_2,\\
					&\|W^{(l)}(t)-W^{(l)}(0)\|_F \leq \omega\sqrt{m}.
				\end{split}
			\end{align}
		\end{lemma}
		
		\begin{proof}
			\begin{align}
				\begin{split}
					\|W^{(l)}(t)-W^{(l)}(0)\|_F &= \|\int_0^t\frac{dW^{(l)}(\tau)}{d\tau}d\tau\|_F\\
					&= \|\int_0^t\frac{\partial L(\tau)}{\partial W^{(l)}(\tau)}d\tau\|_F\\
					&= \|\int_0^t\frac{1}{n}\sum_{i=1}^n(\tilde{v}_i(\tau)-v_i)\mathcal{A}[\frac{\partial h(\theta(\tau),y_i)}{\partial W^{(l)}} ](x_i) d\tau\|_F\\
					& \leq C\frac{\|A\|_\infty}{n}\max_{0\leq\tau\leq t}\sum_{i=1}^n\|\frac{\partial h_{nn}(\theta(\tau),y_i)}{\partial W^{(l)}}\|_F\int_0^t\|{\bf \tilde{v}}_{nn}(\tau)-{\bf v}\|_2d\tau.
				\end{split}
			\end{align}
			
			Thus by Lemma F.7 in \cite{arora2019exact}, we can prove for our BINet, for any $t$ and $l$, the following also holds
			\begin{equation}
				\|W^{(l)}(t)-W^{(l)}(0)\|_F = \mathcal{O}(\sqrt{\frac{n}{\lambda_0}}) \leq \omega\sqrt{m}. 
			\end{equation}
		\end{proof}

		Till now we have verified the lazy properties are satisfied during training, by Lemma \ref{lemma.kernel} and \ref{lemma.reduction}, we complete the proof of Theorem 4.2.
		
	\end{proof}

	\section{Repeated Experiments with Random Initialization}
	With the limit of the number of pages, only some experimental results have been shown in the paper. To better exhibit the accuracy of the experiments, we repeated each experiment of all methods 5 times and show the all relative $L^2$ error. For each experiment, in addition to the initialization of network parameters, other conditions like the optimizer and learning rate are consistent. 
	
	\textbf{Lapalace Equation with Smooth Boundary Condition.}
	In the first experiment, the PINN, Deep Ritz and BINet methods are used respectively to solve the following equaiton
	\begin{equation}
		\begin{aligned}
			-\Delta u(x,y) = 0, \ (x,y)\in\Omega,\\
			u(x,y) = e^{ax}\sin(ay), \ (x,y)\in\partial\Omega.
		\end{aligned}
		\label{Ex1_eq_app}
	\end{equation}
	where $a$ is taken as $4$ and $8$ for the low and high scales.
	The experiments are repeated five times at each setup, i.e., we have done totally $5\times 3\times 2$ times training. Figure \ref{fig:app1} shows all the results of relative $L^2$ error. It is easily to conclude that our BINet is the best and outperforms related methods by a significant margin.
	
	\begin{figure}[t]
		\centering
		\subfigure[a=4.]{
			\centering
			\includegraphics[width=2.5in]{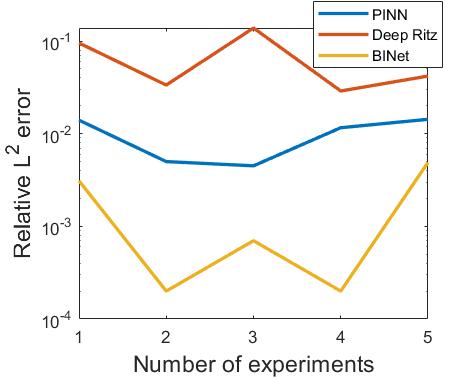}
		}%
		\subfigure[a=8.]{
			\centering
			\includegraphics[width=2.5in]{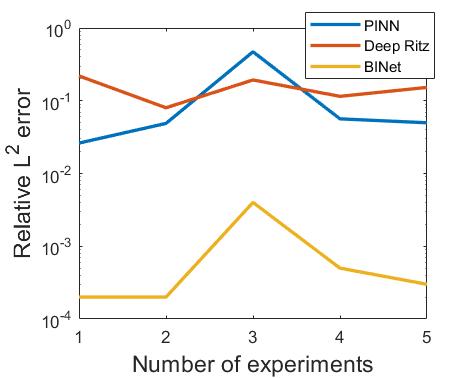}
		}%
		
		\centering
		\caption{The relative $L^2$ error equation of \eqref{Ex1_eq_app} by using PINN, Deep Ritz, and BINet five repeated experiments. The left figure shows the results of $a=4$, and the right one shows the results of $a=8$.}
		\label{fig:app1}
	\end{figure}

	\textbf{Laplace Equation with Non-smooth Boundary Condition.} In this experiment, we consider a Laplace equation problem
	\begin{equation}
		\begin{aligned}
			-\Delta u(x,y)=0,\ (x,y)\in\Omega,\\
			u(x,y) = 2-|x|-|y|, \ (x,y)\in\partial\Omega.
		\end{aligned}
		\label{app_eq2}
	\end{equation}
	with the non-smooth boundary condition.  The performance of the PINN, Deep Ritz, and BINet has been shown in the main part of the paper. In this appendix, the numerical experiments using each method are run five times and the results of all relative $L^2$ error are shown in Figure \ref{fig:app2}. It is easy to find that the results using BINet have much higher accuracy than others with more than 100x decrease of the $L^2$ error at most. Moreover, the numerical stability of BINet is also much better.
	
	\begin{figure}[t]
		\centering
		\centering
		\includegraphics[width=0.6\textwidth]{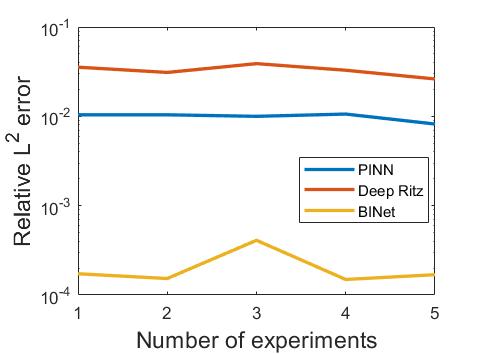}
		\centering
		\caption{The relative $L^2$ error equation of \eqref{app_eq2} by using PINN, Deep Ritz, and BINet of 5 repeated experiments.}
		\label{fig:app2}
	\end{figure}

	\textbf{Helmholtz Equation with Different Wave numbers.} In the main part of the paper, we have considered a Helmholtz equation with different wave numbers. The Helmholtz equation problem has following form
	\begin{equation}
		\begin{aligned}
			-\Delta u(x,y) - k^2u(x,y) = 0,\ (x,y)\in\Omega,\\
			u(x,y) = e^{i(k_1x+k_2y)}, \ (x,y)\in\partial\Omega,
		\end{aligned}
		\label{eq:Ex3_app}
	\end{equation}
	where $(k_1,k_2)=(k\cos \frac{\pi}{7},k\sin\frac{\pi}{7})$, and the boundary has the parametric representation $\partial\Omega=\{(\frac{9}{20}\cos(t)-\frac{\cos(5t)}{9}\cos(t),\frac{9}{20}\sin(t)-\frac{\cos(5t)}{9}\sin(t))|t\in[0,2\pi]\}$. In the experiment, we assume $k=1$ and $4$. We also repeated the experiments 5 times for both PINN and BINet methods. The relative $L^2$ error of each experiment has shown in Figure \ref{fig:app3}.

	\begin{figure}[t]
		\centering
		\subfigure[k=1.]{
			\centering
			\includegraphics[width=2.5in]{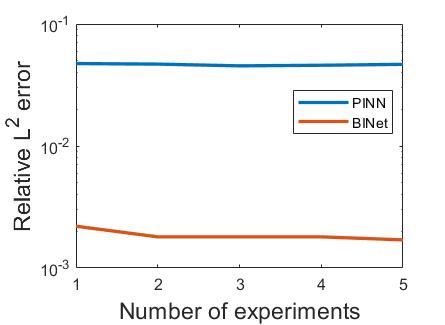}
		}%
		\subfigure[k=4.]{
			\centering
			\includegraphics[width=2.5in]{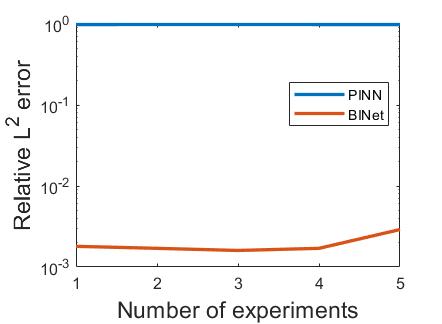}
		}%
		
		\centering
		\caption{The relative $L^2$ error equation of \eqref{eq:Ex3_app} by using PINN and BINet of 5 repeated experiments. The left figure shows the results of $k=1$, and the other shows the result of $k=4$.}
		\label{fig:app3}
	\end{figure}
	In Figure \ref{fig:app3} we can see the results of repeated experiments are similar. This also confirms the stability of our method. We know the difficulty of the Helmholtz equation will increase with the increase of the wave number. PINN method failed when the wave number $k=4$, but the relative $L^2$ error of the BINet method has no obvious difference between $k=4$ and $k=1$.

	\noindent \textbf{The Operator from Equation Parameters to Solutions.} In the main part of the paper, we have verified the ability of BINet to learn operators and solve PDEs on the unbounded domain. We also tested the generalization ability of BINet. We consider the Helmholtz equations with various wave numbers.
	\begin{equation}
		\begin{aligned}
			-\Delta u(x,y) - k^2u(x,y) = 0,\ (x,y)\in\Omega,\\
			u(x,y) = H_0^1(k\sqrt{x^2+y^2}), \ (x,y)\in\partial\Omega,
		\end{aligned}
		\label{eq:app4}
	\end{equation}
	In the training phase, we set $k\in[2,3.5]\cup[4.5,6]$.
	Figure \ref{fig:app4} shows the relative $L^2$ error of different wave numbers. Five lines of different colors represent the results of five repeated experiments. We can see the five results are similar, too. When BINet learns an operator to solve a parametric PDE, this method still has good numerical stability. 
	
	\begin{figure}[t]
		\centering
		\centering
		\includegraphics[width=0.6\textwidth]{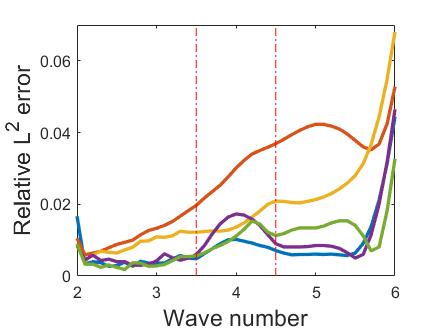}
		\centering
		\caption{The relative $L^2$ error equation of \eqref{eq:app4} with different wave numbers by using and BINet of five repeated experiments.}
		\label{fig:app4}
	\end{figure}

	\noindent \textbf{The Operator from Boundary Geometry to Solutions.} In the final experiment of the main paper, we used the BINet method to learn a operator from boundary  geometry to solutions. We consider a Laplace equation with parametric boundaries. The problem is 
	\begin{equation}
		\begin{aligned}
			-\Delta u(x,y) & = 0, \ &(x,y)\in\Omega_\beta,\\
			u(x,y) &= g(x,y;\beta), \ &(x,y)\in\partial\Omega_\beta,
		\end{aligned}
		\label{eq:ex5_app}
	\end{equation}
	where the boundary condition $g(x,y;\beta)=(x-x_{bc})(y-y_{bc})+(x-x_{bc})+(y-y_{bc})+1$, and $(x_{bc},y_{bc})$ is the barycenter of the $\Omega_\beta$. For details of the parametric boundary, please refer to the main paper. We repeated the experiment five times by using the BINet method. After each experiment, we selected 100 triangles and calculate the relative $L^2$ error for each triangle. Figure \ref{fig:app5} shows the average relative $L^2$ error of the 100 triangles of each experiment.

	\begin{figure}[t]
		\centering
		\centering
		\includegraphics[width=0.6\textwidth]{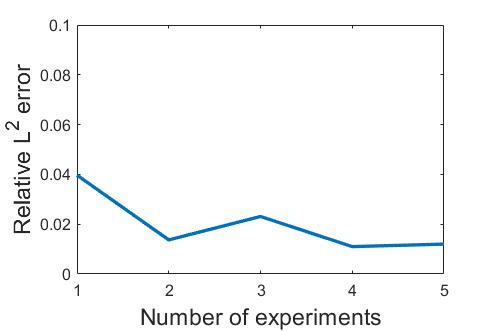}
		\centering
		\caption{The relative $L^2$ error equation of \eqref{eq:ex5_app} by using BINet of five repeated experiments. The relative $L^2$ error of each experiment is the average over 100 triangles.}
		\label{fig:app5}
	\end{figure}
	
	\section{More Experiments}
	
	\textbf{Comparison of single layer potential and double layer potential.} In this experiment, we use the single layer potential and double layer potential to construct BINet, respectively. We also consider this Helmholtz equation with wave number $k=4$,
	\begin{equation}
		\begin{aligned}
			-\Delta u(x,y) - 16u(x,y) = 0,\ (x,y)\in\Omega,\\
			u(x,y) = e^{i(k_1x+k_2y)}, \ (x,y)\in\partial\Omega,
		\end{aligned}
		\label{eq:Ex6}
	\end{equation}
	where $\partial\Omega$ has the same parametric representation as the previous boundary of the problem \eqref{eq:Ex3_app}, and in this example, we also assume $(k_1,k_2)=(4\cos\frac{\pi}{7},4\sin\frac{\pi}{7})$.
	We can know the exact solution of this problem is 
	\begin{equation*}
		u(x,y) = e^{i(k_1x+k_2y)}, \ (x,y)\in\Omega.
	\end{equation*}
	We select 600 sample points on the boundary to construct the loss function, and we use the Adam optimizer with learning rate 0.0001. After 40000 training epochs, we random 1000 points in $\Omega$. By comparing with the exact solution, the relative $L^2$ error is approximate 0.0086 for single layer potential and 0.0016 for double layer potential, and we have the following results of the loss function and error.
	\noindent
	\begin{figure}[t]
		\centering
		\includegraphics[width=0.31\textwidth]{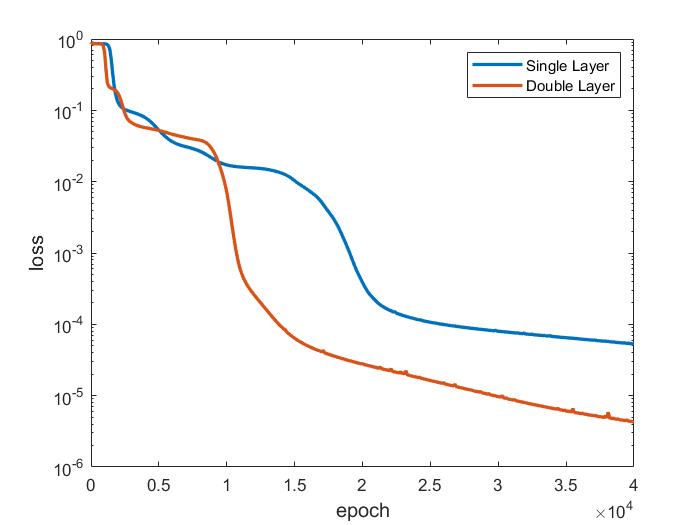}
		\includegraphics[width=0.32\textwidth]{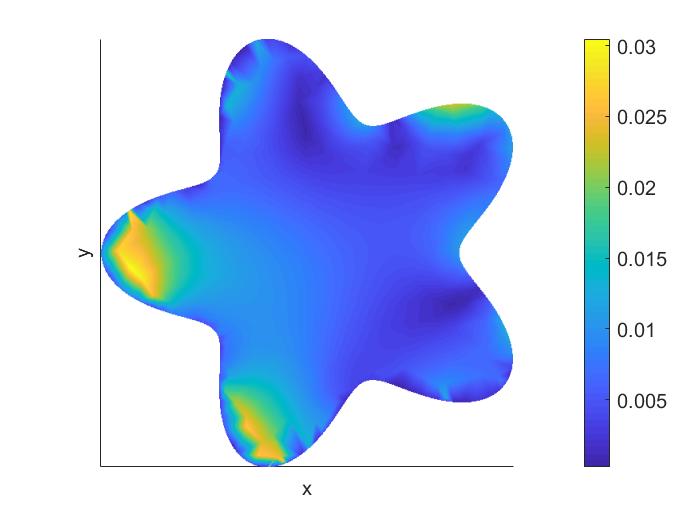}
		\includegraphics[width=0.32\textwidth]{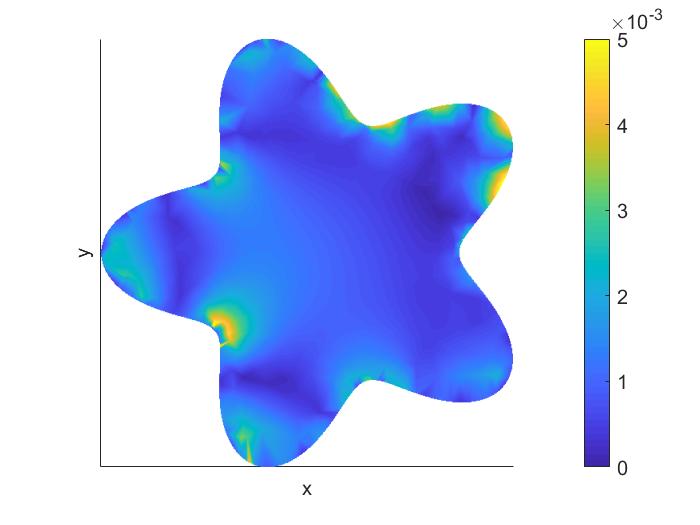}
		\caption{The results of the equation \eqref{eq:Ex6}. The first figure is the loss function of BINet with single and double layer potentials respectively. The two figures on the right are the absolute error between the numerical solutions and the exact solution with single and double layer potentials respectively. }
		\label{fig:ex6}
	\end{figure}
	From Figure \ref{fig:ex6}, we can see the solution of the double layer potential is better than the solution of the single layer potential. In fact, because of the jump of the double layer potential at the boundary, the condition number of the double layer potential is better than the single, and the BINet with double layer potential converges faster.

	\noindent \textbf{Laplace equation on unbounded domain.} We assume the boundary of $\Omega$ is same as the previous experiment. We will use BINet to solve the following equation
	\begin{equation}
		\begin{aligned}
			-\Delta u(x,y) = 0,\ (x,y)\in\Omega^c,\\
			u(x,y) = \frac{x}{x^2+y^2}, \ (x,y)\in\partial\Omega,
		\end{aligned}
		\label{eq:app7}
	\end{equation}
	The exact solution of this example is 
	$$
	u(x,y) = \frac{x}{x^2+y^2},\ (x,y)\in\Omega^c.
	$$
	In this experiment, the double layer potential is used to construct the loss function. Adam optimizer is used in the optimization with the learning rate of 0.001. We selected 600 sample points on the boundary. After 20000 training epochs, the relative $L^2$ error is about 0.0030, and the error map is shown in Figure \ref{fig:eq7}. This experiment verifies the feasibility of BINet to solve the exterior Laplace equations on the unbounded domain.  
	
	\noindent
	\begin{figure}[t]
		\centering
		\includegraphics[width=0.5\textwidth]{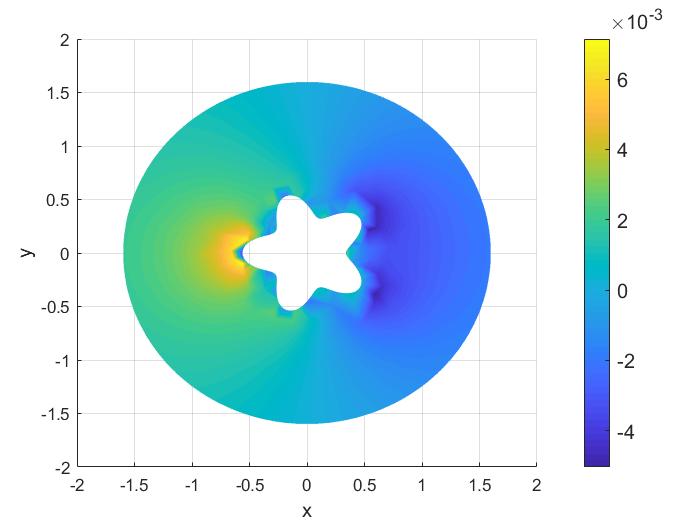}
		\caption{The result of the equation \eqref{eq:app7}. The figure shows the error between the solution of BINet and the exact solution on the domain $\{(x,y)|x^2+y^2<1.6^2\}\cap\Omega^c$.}
		\label{fig:eq7}
	\end{figure}
	
	\noindent \textbf{Helmholtz equations with high wave numbers.}
	In this experiment, we consider an interior problem of Helmholtz equation on a bounded domain with wave number $k=10$.
	Let us consider this Helmholtz equation in an interior domain
	\begin{equation}
		\begin{aligned}
			-\Delta u(x,y) - 100u(x,y) = 0,\ (x,y)\in\Omega_1,\\
			u(x,y) = e^{i(k_1x+k_2y)}, \ (x,y)\in\partial\Omega_1,
		\end{aligned}
		\label{eq:ex8_1}
	\end{equation}
	where $(k_1,k_2)=(10\cos\frac{\pi}{5},10\sin\frac{\pi}{5})$,
	and in an exterior domain with wave number $k=8$
	\begin{equation}
		\begin{aligned}
			-\Delta u(x,y) -64u(x,y) = 0,\ (x,y)\in\Omega_2^c,\\
			u(x,y) = H_0^1(8\sqrt{(x-0.5)^2+y^2}), \ (x,y)\in\partial\Omega_2,
		\end{aligned}
		\label{eq:ex8_2}
	\end{equation}
	where $H_0^1$ is the first kind Hankel function. The boundaries $\partial\Omega_1$ and $\partial\Omega_2$ are shown in Figure \ref{fig:app8}. By BINet, the relative $L^2$ errors of solutions of the equations \eqref{eq:ex8_1}
	and \eqref{eq:ex8_2} are $1.5\%$ and $2.8\%$, respectively. 
	\noindent
	\begin{figure}[H]
		\centering
		
		\includegraphics[width=0.48\textwidth]{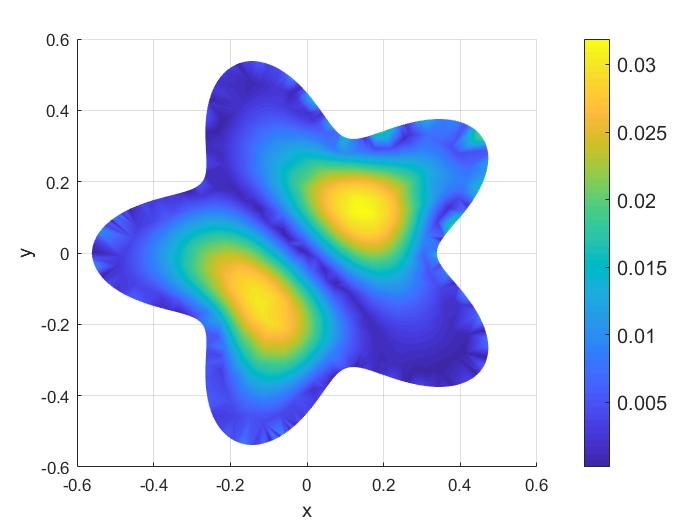}
		\centering
		\includegraphics[width=0.48\textwidth]{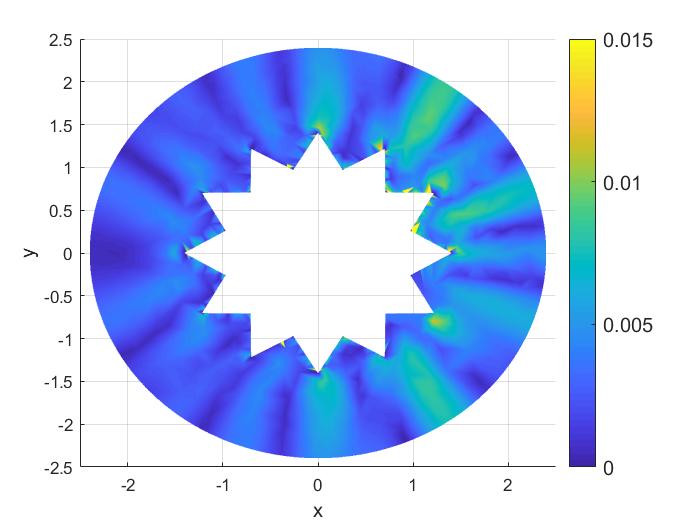}
		\caption{The results of equations \eqref{eq:ex8_1} and \eqref{eq:ex8_2}. The left figure is the absolute error of the equation \eqref{eq:ex8_1} and the other is the absolute error of the equation \eqref{eq:ex8_2}.}
		\label{fig:app8}
	\end{figure}

\end{document}